\documentclass[a4j,11pt]{article}
\usepackage[top=25truemm,bottom=25truemm,left=29truemm,right=29truemm]{geometry}
\usepackage{latexsym}
\usepackage{amsmath}
\usepackage{amssymb}
\usepackage{mathrsfs}
\usepackage{bm}
\usepackage[pdftex]{graphicx}
\usepackage{wrapfig}
\usepackage{fancybox}
\usepackage{ascmac}
\usepackage{amsthm}
\usepackage{tikz}
\usetikzlibrary{automata}
\numberwithin{equation}{section}

\usepackage{caption}

\newtheoremstyle{mystyle}
{} 
{} 
{\it} 
{} 
{\bf} 
{} 
{ } 
{} 

\theoremstyle{mystyle}
\newtheorem{Thm}{Theorem}[section] 
\newtheorem{Obs}{Observation}[section]
\newtheorem{Cor}[Thm]{Corollary}
\newtheorem{Lem}[Thm]{Lemma}
\newtheorem{Prop}[Thm]{Proposition}
\newtheorem{Def}[Thm]{Definition}

\newtheorem{Remark}{Remark}[section]
\newtheorem{Example}{Example}[section]

\newcommand{\lbc}{[{\hspace{-0.15em}[}}
\newcommand{\rbc}{]{\hspace{-0.15em}]}}
\newcommand{\wk}{\mbox{{\textit{1}}{\hspace{-0.3em}$k$}}}
\newcommand{\wl}{\mbox{{$\ell$}{\hspace{-0.27em}{$l$}}}}

\newcommand{\ba}{\mbox{\boldmath $a$}}
\newcommand{\bc}{\mbox{\boldmath $c$}}

\newcommand{\br}{\mbox{\boldmath $r$}}

\newcommand{\bzero}{\mbox{\boldmath $0$}}

\usepackage{multicol}

\title{Finite beta-expansions of natural numbers}
\author{Fumichika Takamizo}
\date{\empty}

\begin{document}

\maketitle

\begin{abstract}
Let $\beta>1$. 
For $x \in [0,\infty)$, we have so-called the {\it{beta-expansion}} 
of $x$ in base $\beta$ as follows: 
$$x= \sum_{j \leq k} x_{j}\beta^{j} 
= x_{k}\beta^{k}+ \cdots + x_{1}\beta+x_{0}+x_{-1}\beta^{-1} 
+ x_{-2}\beta^{-2} + \cdots$$
where $k \in \mathbb{Z}$, $\beta^{k} \leq x < \beta^{k+1}$, 
$x_{j} \in \mathbb{Z} \cap [0,\beta)$ for all $j \leq k$ and 
$\sum_{j \leq n}x_{j}\beta^{j}<\beta^{n+1}$ for all $n \leq k$. 
In this paper, we give a sufficient condition (for $\beta$) such that 
each element of $\mathbb{N}$ has the finite beta-expansion in base $\beta$. 
Moreover we also find a $\beta$ with this finiteness property 
which does not have positive finiteness property. 
\end{abstract}

\footnote[0]{This work was partly supported by MEXT Promotion of Distinctive 
Joint Research Center Program JPMXP0619217849.}

\section{Introduction}

When $\beta>1$ is an integer, it is readily seen that each element of $\mathbb{Z}[1/\beta] \cap [0, \infty)$ has a finite expansion in base $\beta$. As a generalization of this property, Frougny and Solomyak proposed three kinds of finiteness properties (F), (PF) and (F$_{1}$) described in detail later. The property (F$_{1}$), which is the condition that each natural number has a finite beta-expansion in base $\beta$, is not yet well-understood. 
In this paper, we find a $\beta$ with property (F$_{1}$) and without property (PF). 

We describe notions and definitions in more detail. 
For $\beta>1$, define the $\beta$-transformation $T:[0,1] \to [0,1)$ by $T(x) = \{\beta x\}$ where $\{y\}$ is the fractional part of $y$. Let $x_{0}=x \in [0,1]$, $x_{n}=T(x_{n-1})$ and $c_{n}=\lfloor \beta x_{n-1}\rfloor$ where $\lfloor y\rfloor$ is the integer part of $y$. Then we have the expansion of $x$, that is, 
$$x = c_{1}\beta^{-1}+c_{2}\beta^{-2}+\cdots + c_{n}\beta^{-n}+\cdots = \sum_{n=1}^{\infty} c_{n}\beta^{-n},$$
and write 
$$d_{\beta}(x)=c_{1}c_{2}\cdots c_{n}\cdots.$$
In particular, $d_{\beta}(1)$ is called the {\it{R\'enyi expansion}}. 
Since $T(1) = \beta-\lfloor \beta \rfloor = \{\beta\}$, 
we have 
$$d_{\beta}(1) = \lfloor \beta \rfloor d_{\beta}(\{\beta\}).$$

Now we give the definition of the beta-expansion of $x \geq 0$. 
Let $\mathbb{N}_{0}=\mathbb{N} \cup \{0\}$. 
Define $L:[0,\infty) \to \mathbb{N}_{0}$ by 
$$L(x):= \min\{n\in \mathbb{N}_{0} \ | \ x \beta^{-n}<1\}.$$
So $x \in [0,1)$ if and only if $L(x) = 0$. In this paper, 
{\it{the beta-expansion of $x$}} is defined as follows: for $x \in [0,\infty)$, 
the beta-expansion of $x$ is given by 
$$x= c_{1}\beta^{L(x)-1}+\cdots + c_{L(x)-1}\beta+c_{L(x)} 
+ c_{L(x)+1}\beta^{-1}+\cdots  
= \sum_{n=1}^{\infty} c_{n}\beta^{L(x)-n}$$
where $d_{\beta}(\beta^{-L(x)}x) = c_{1}c_{2} \cdots$. 
In other articles, the usual beta-expansion of $x$ is given by 
\begin{equation*}
\begin{split}
x &= x_{k}\beta^{k}+ \cdots + x_{1}\beta+x_{0}+x_{-1}\beta^{-1} 
+ x_{-2}\beta^{-2} + \cdots = \sum_{n \leq k}x_{n}\beta^{n} \\
&= x_{k} \cdots x_{0}.x_{1} \cdots \ \mbox{(symbolically)} \\
\end{split}
\end{equation*}
where $k=k(x) \in \mathbb{Z}$, $x_{k}>0$, $x_{j} \in [0,\beta) \cap \mathbb{Z}$ for all $j \leq k$ and $\sum_{j \leq n}x_{j}\beta^{j}<\beta^{n+1}$ for all $n \leq k$. For example, 
$$\mbox{the beta-expansion of $1$}=1.0^{\infty}.$$
In contrast, $L(x) = 0\geq k(x)$ if $x<1$; $L(x)-1=k(x)$ if $x\geq 1$. 
Moreover $c_{L(x)-n} = x_{n}$ for all $n \leq k(x)$. Clearly, we have 
$$x=\beta^{L(x)}\sum_{n=1}^{\infty}c_{n}\beta^{-n}.$$
So our notation of the beta-expansion is like a {\it{floating-point representation}}. 

For $x \in [0,1]$, $d_{\beta}(x)=c_{1}c_{2}\cdots$ is called {\it{finite}} if $\#\{ n \in \mathbb{N} \ | \ c_{n}\neq 0\} < \infty$. We say that $x \geq 0$ has the {\it{finite beta-expansion}} if $d_{\beta}(\beta^{-L(x)}x)$ is finite. Denote by $\mbox{Fin}(\beta)$ the set of nonnegative number $x$ such that $x$ has the finite beta-expansion. 

\begin{wrapfigure}{r}{45mm}
\centering
\caption{{The inclusion relations of finiteness properties}}\label{inclusion}
\begin{tikzpicture}
\node(q0) at (0,0) {(F)};
\node(q1) at (0,0.8) {(PF)};
\node(q2) at (0,1.5) {(F$_{1}$)};
\node(q3) at (0,1.95) {}; 
\draw(0,0) circle(0.5);
\draw(0,0.32) circle(0.95); 
\draw(0,0.55) circle(1.35);
\end{tikzpicture}
\end{wrapfigure}
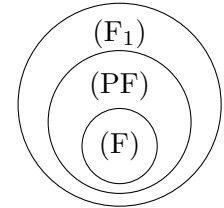
If $\beta>1$ is a natural number, then 
$\mbox{Fin}(\beta) = \mathbb{Z}[1/\beta] \cap [0,\infty)$. 
In \cite{frougny1}, Frougny and Solomyak studied the problem of what $\beta$ satisfies this condition and more generally, introduced the following conditions: 
\begin{equation*}
\begin{array}{ll}
\mbox{(F$_{1}$)} & \mathbb{N}_{0} \subset \mbox{Fin}(\beta) \\
\mbox{(PF)} & \mathbb{N}_{0} [1/\beta] \subset \mbox{Fin}(\beta) \\
\mbox{(F)} & \mathbb{Z}[1/\beta] \cap [0,\infty) \subset \mbox{Fin}(\beta) \\
\end{array}
\end{equation*}
By definition, if $\beta$ has property (F), then $\beta$ has property (PF); if $\beta$ has property (PF), then $\beta$ has property (F$_{1}$) (see Figure \ref{inclusion}). 

According to \cite{akiyama1,frougny1}, 
if $\beta$ has property (F$_{1}$), 
then $\beta$ is a Pisot number. 
Here an algebraic integer $\beta$ is called a Pisot number 
if $\beta>1$ and all conjugates other than $\beta$ 
have modulus less than $1$. 

The property (F) is also called the {\it{finiteness property}}. 
Several sufficient conditions for property (F) are known. 
Let $\beta$ be an algebraic integer with minimal polynomial 
$x^{d}-a_{d-1}x^{d-1} - \cdots -a_{1}x-a_{0}$. 
If $a_{d-1} \geq \cdots \geq a_{1} \geq a_{0} \geq 1$, 
then $\beta$ has property (F) and this $\beta$ is called of Frougny-Solomyak type 
(\cite{frougny1}). 
If $a_{d-1} > a_{d-2} + \cdots +a_{0}$ and $a_{j} \geq 0$ for all $j$, 
then $\beta$ has property (F) and this $\beta$ is 
called of Hollander type (\cite{hollander}). 
Moreover in \cite{akiyama1}, Akiyama characterized cubic Pisot units with property (F). 

The property (PF) is also called the {\it{positive finiteness property}}. 
In \cite{frougny1}, it is shown that 
each Pisot number $\beta$ with degree $2$ has property (PF). 
Furthermore in \cite{akiyama2}, 
Akiyama gave complete characterization of property (PF) without (F). 

Bassino has found $d_{\beta}(1)$ for all cubic Pisot numbers (\cite{bassino}). 
Moreover we remark the following proposition. 

\begin{Prop}\label{CPcase}
Let $\beta$ be a cubic Pisot number with property (F$_{1}$). Then 
$\beta$ has property (PF) without (F) if and only if 
$d_{\beta}(1)$ is not finite. 
\end{Prop}

The main purpose in this paper is to prove the following theorem. 

\begin{Thm}\label{main1}
Let $\beta>1$ be the root of $x^{3}-2tx^{2}+2tx-t$ 
$(t \in \mathbb{N} \cap [2,\infty))$. 
Then $\beta$ has the property (F$_{1}$) without (PF). 
\end{Thm}

This paper is structured in the following way. In section \ref{sec:key}, we define the fractional part of the beta-expansion of $x$ (so-called {\it{$\beta$-fractional part of $x$}}) by 
$$\{x\}_{\beta} = \sum_{n=1}^{\infty} c_{L(x)+n}\beta^{-n} = c_{L(x)+1}\beta^{-1}+c_{L(x)+2}\beta^{-2}+\cdots$$
and describe $\{x+1\}_{\beta}-\{x\}_{\beta}$ by using $T^{n}(1)$ $(n \in \mathbb{N}_{0})$. In section \ref{sec:main2}, we give a sufficient condition for property (F$_{1}$) (see Theorem \ref{main2}). In this section, we use the transformation $\tau$ on $\mathbb{Z}^{d-1}$ introduced in \cite{akiyama3,akiyama4,hollander}, which is conjugate to the transformation $T$ on $\mathbb{Z}[\beta] \cap [0,1)$, when $\beta$ is an algebraic integer with degree $d$. In section \ref{sec:main1}, we prove Theorem \ref{main1} by using the sufficient condition in section \ref{sec:main2}. Finally we provide the proof of Proposition \ref{CPcase}. 

\section{Preliminary}\label{sec:key}

Let $\beta>1$ be a real number and define $\nu$ by 
$$\nu(c_{1}c_{2} \cdots ) = \sum_{n=1}^{\infty} c_{n}\beta^{-n}$$
where $c_{1}c_{2} \cdots \in \mathbb{Z}$ satisfies 
$\sum_{n\geq 1} |c_{n}|\beta^{-n} < \infty$. 
For $x \in [0,\infty)$, 
define {\it{$\beta$-fractional part of $x$}} by 
$$\{x\}_{\beta} = \nu(c^{}_{L(x)+1}c^{}_{L(x)+2} \cdots) \ \mbox{where} \ 
d_{\beta}(\beta^{-L(x)}x)=c_{1}c_{2} \cdots.$$
We aim to prove the following theorem. 

\begin{Thm}\label{key}
For any $x \in [0,\infty)$, 
there exist $\theta \in \{0,1\}$, $q \in \mathbb{N}$ 
and $\omega_{j} \in \mathbb{N}_{0}$ $(j \in \lbc0,q\rbc)$ such that 
$$\{x+1\}_{\beta}-\{x\}_{\beta} 
= \theta - \sum_{j=0}^{q} \omega_{j} T^{j}(1).$$
\end{Thm}

\subsection{Fundamental property of beta expansion}\label{subsec:fund_lem}

In this subsection, we collect well-known results and notions. 
For $m,n \in \mathbb{Z}$ with $m\leq n$, 
let $\lbc m,n \rbc = [m,n] \cap \mathbb{Z}$ and 
\begin{equation*}
\mathbb{A} := \left\{ 
\begin{array}{ll}
\lbc 0,\beta-1 \rbc & \mbox{if $\beta$ is an integer} \\
\lbc 0,\lfloor \beta \rfloor\rbc & \mbox{if $\beta$ is not an integer.} \\
\end{array}
\right. 
\end{equation*}
Define {\it{the shift}} $\sigma : \mathbb{A}^{\mathbb{N}} \to \mathbb{A}^{\mathbb{N}}$ by 
$$\sigma(c_{1}c_{2} \cdots ) = c_{2}c_{3} \cdots.$$
Note that $\sigma$ is continuous on $\mathbb{A}^{\mathbb{N}}$. 
Moreover by definition of $T$, we have 
$$\sigma(d_{\beta}(x)) = d_{\beta}(T(x)) \ \mbox{for $x \in [0,1]$}.$$
That is, we get the following commutative diagram: 
\begin{equation*}
\begin{array}{ccc}\vspace{0.1cm}
[0,1] & \overset{T}{\longrightarrow} & [0,1] \\ \vspace{0.1cm}
d_{\beta} \Big\downarrow \quad & & \quad \Big\downarrow d_{\beta} \\
\mathbb{A}^{\mathbb{N}} & \overset{\sigma}{\longrightarrow} & \mathbb{A}^{\mathbb{N}}. \\
\end{array}
\end{equation*}
We often use the notation $\bc=c_{1}c_{2} \cdots$. 
Denote by $<_{lex}$ the lexicographic order on $\mathbb{N}_{0}^{\mathbb{N}}$. 
For $\bc,\bc' \in \mathbb{N}_{0}^{\mathbb{N}}$, 
$\bc \leq_{lex} \bc'$ means $\bc <_{lex} \bc'$ or $\bc=\bc'$. 
Then $d_{\beta}$ is order preserving, that is, 

\begin{Remark}[(\cite{lothaire})]\label{lothaire}
Let $x,y \in [0,1]$. 
$x<y$ if and only if 
$d_{\beta}(x)<_{lex}d_{\beta}(y)$. 
\end{Remark}

The metric $\rho$ on $\mathbb{N}_{0}^{\mathbb{N}}$ is defined by 
$$\rho(\bc, \bc') 
= (\inf \{n \in \mathbb{N} \ | \ c_{n} \neq c_{n}'\})^{-1}.$$
Note that $\nu$ is continuous on $\mathbb{A}^{\mathbb{N}}$. 
Define 
$$d_{\beta}^{*}(x) := \lim_{y \uparrow x}d_{\beta}(y) \ 
\mbox{for each $x \in (0,1]$.}$$
Then $\sigma(d_{\beta}^{*}((0,1])) \subset d_{\beta}^{*}((0,1])$. 
Let 
$$D_{\beta}: = d_{\beta}([0,1)).$$
By using $d_{\beta}^{*}(1)$, 
we can determine whether or not any 
$c_{1}c_{2}\cdots \in \mathbb{N}_{0}^{\mathbb{N}}$ is in $D_{\beta}$. 

\begin{Thm}[(\cite{parry1,itotaka1})]\label{parry}
$\bc \in \mathbb{N}_{0}^{\mathbb{N}}$ 
belongs to $D_{\beta}$ if and only if 
$\sigma^{n}(\bc) <_{lex} d_{\beta}^{*}(1)$ 
for all $n \in \mathbb{N}_{0}$. 
\end{Thm}

Now we have the following. 

\begin{Remark}[(\cite{lothaire} Chapter 7, cf. \cite{mine1})]\label{d_1}
\begin{equation*}
d_{\beta}^{*}(1) = \left\{ 
\begin{array}{ll}
(d_{1}\cdots d_{q-1}(d_{q}-1))^{\infty} & \mbox{if $d_{\beta}(1)=d_{1}d_{2} \cdots d_{q}0^{\infty}$} \\
d_{\beta}(1) & \mbox{otherwise.} \\
\end{array}
\right. 
\end{equation*}
\end{Remark}

Since $\nu(d_{\beta}(y))=y$ for $y \in [0,1)$, we have by the continuity of $\nu$, 
$$\nu(d_{\beta}^{*}(x)) = \lim_{y \uparrow x} \nu(d_{\beta}(y)) =x.$$
Thus $d_{\beta}^{*}$ is injective on $(0,1]$ (cf. \cite{mine1}). Let 
$$\xi(n) := \nu(\sigma^{n-1}(d_{\beta}^{*}(1))) \ (n \in \mathbb{N}).$$

\begin{Obs}\label{obs1}
$d_{\beta}^{*}(\xi(n)) = \sigma^{n-1}(d_{\beta}^{*}(1))$. 
\end{Obs}

\begin{proof}
Notice that $\sigma^{n-1}(d_{\beta}^{*}(1)) \in d_{\beta}^{*}((0,1])$ because $\sigma(d_{\beta}^{*}((0,1]) ) \subset d_{\beta}^{*}((0,1])$. Hence by definition of $\xi(n)$ and the injectivity of $d_{\beta}^{*}$, we get the assertion. 
\end{proof}

In the following, we write 
$$d_{\beta}^{*}(1):=d_{1}d_{2} \cdots.$$

\begin{Lem}\label{fund_lem}
 \ 
\begin{description}
\item{(1)} \ $\{ \xi(n) \ | \ n \in \mathbb{N}\} 
= \{ T^{m}(1) \ | \ m \in \mathbb{N}_{0}\} \setminus \{0\}$. 
\item{(2)} \ $d_{n}d_{n+1}\cdots \leq_{lex} d_{\beta}^{*}(1)$ 
for all $n \in \mathbb{N}$. 
\item{(3)} \ If $\bc=c_{1}c_{2}\cdots \in \mathbb{N}_{0}^{\mathbb{N}}$ 
satisfies $\bc \geq_{lex} d_{n}d_{n+1}\cdots$, then 
$\nu(\bc) \geq \xi(n)$. 
\item{(4)} \ If $\bc=c_{1}c_{2}\cdots \in D_{\beta}$ 
satisfies $\bc <_{lex} d_{n}d_{n+1}\cdots$, then 
$\nu(\bc) < \xi(n)$. 
\end{description}
\end{Lem}

\begin{proof}
(1) If $d_{\beta}(1)=d_{\beta}^{*}(1)$, then $T^{n-1}(1) \neq 0$ and 
so by Observation \ref{obs1}, 
$$d_{\beta}^{*}(\xi(n)) = \sigma^{n-1}(d_{\beta}^{*}(1)) 
= \sigma^{n-1} d_{\beta}(1) = d_{\beta}(T^{n-1}(1)).$$
Therefore in this case, 
$$\xi(n) = \nu(d_{\beta}^{*}(\xi(n))) 
= \nu(d_{\beta}(T^{n-1}(1))) = T^{n-1}(1).$$
Consider the case $d_{\beta}(1) \neq d_{\beta}^{*}(1)$. 
Then by Remark \ref{d_1}, there is $q \in \mathbb{N}$ 
such that $T^{q}(1) = 0$ and $T^{q-1}(1)>0$. 
Therefore for any $n \in \lbc1,q\rbc$, 
$$d_{\beta}^{*}(\xi(n)) = d_{\beta}^{*}(T^{n-1}(1))$$
and so by taking $\nu$, 
$$\xi(n) = T^{n-1}(1).$$
On the other hand, if $n >q$, 
then 
$$d_{\beta}^{*}(\xi(n)) = d_{\beta}^{*}(\xi(k)) 
= d_{\beta}^{*}(T^{k-1}(1))$$
and so $\xi(n) = T^{k-1}(1)$ 
where $n \equiv k \mod{q}$. 

(2) By (1), there is $m \in \mathbb{N}_{0}$ such that $\xi(n) = T^{m}(1)$. 
Thus $\xi(n) \leq 1$. 
So, since $d_{\beta}(\xi(n)-\epsilon) \leq_{lex} d_{\beta}(1-\epsilon)$ 
by Remark \ref{lothaire}, we have 
$$d_{\beta}^{*}(\xi(n)) = 
\lim_{\epsilon \downarrow 0} d_{\beta}(\xi(n)-\epsilon) 
\leq_{lex} \lim_{\epsilon \downarrow0} d_{\beta}(1-\epsilon) 
\leq_{lex} d_{\beta}^{*}(1).$$

(3) It suffices to show that 
if $\bc >_{lex} d_{n}d_{n+1} \cdots$, then $\nu(\bc)\geq \xi(n)$. Let 
$$m:=\min\{ j \in \mathbb{N} \ | \ c_{j} \neq d_{n+j-1} \}.$$
Then $c_{m} > d_{n+m-1}$. So, since 
$c_{1}c_{2}\cdots c_{m-1}=d_{n}d_{n+1}\cdots d_{n+m-2}$, we have 
\begin{equation*}
\begin{split}
&\quad \ \nu(\bc)-\nu(d_{n}d_{n+1}\cdots) \\
&= \beta^{-m+1}(\nu(c_{m}c_{m+1} \cdots)-\nu(d_{n+m-1}d_{n+m} \cdots) ) \\
&\geq \beta^{-m+1}+\beta^{-m+1}(\nu(c_{m+1}c_{m+2} \cdots) 
- \nu(d_{n+m}d_{n+m+1}\cdots) ) \\
&\geq \beta^{-m+1} - \beta^{-m+1}\nu(d_{n+m}d_{n+m+1}\cdots)=:K \\
\end{split}
\end{equation*}
By (1), $\nu(d_{n+m}d_{n+m+1}\cdots) = \xi(n+m) = T^{p}(1)$ 
for some $p \in \mathbb{N}_{0}$. So 
$$K =\beta^{-m} -\beta^{-m}T^{p}(1) \geq 0.$$

(4) Let $\bc <_{lex} d_{n}d_{n+1} \cdots$. 
By Observation \ref{obs1} and definition of $d_{\beta}^{*}$, 
$$\bc <_{lex} d_{\beta}(\xi(n)-\epsilon)$$
for some $\epsilon>0$. Therefore we have 
$$\nu(\bc)<\xi(n)-\epsilon < \xi(n).$$
Hence we get the assertion. 
\end{proof}

Let 
$$\mathscr{X}:= \left\{ \sum_{m=1}^{p}u_{m}\xi(m) \ \Biggl| \ p \in \mathbb{N}, \ u_{m} \in \mathbb{N}_{0} \ (m \in \lbc1,p\rbc) \right\}.$$

As a corollary of Lemma \ref{fund_lem}-(1), we have the following. 

\begin{Remark}\label{fund_rem}
$$\mathscr{X} = \left\{ \sum_{j=0}^{q}\omega_{j}T^{j}(1) \ \Biggl| \ q \in \mathbb{N}, \omega_{j} \in \mathbb{N}_{0} \ (j \in \lbc0,q\rbc) \right\}.$$
\end{Remark}

\begin{proof}
Notice that $\mathscr{X}$ 
is the additive monoid generated by $\{\xi(n) \ | \ n \in \mathbb{N}\}$. 
So by Lemma \ref{fund_lem}-(1), we get the assertion. 
\end{proof}

\subsection{Normalization Procedure}\label{subsec:procedure}

For $\bc=c_{1}c_{2} \cdots$, let 
\begin{equation*}
\bc[i,j]:= \left\{ 
\begin{array}{ll}
c_{i}c_{i+1} \cdots c_{j} & \mbox{if $i \leq j<\infty$} \\
c_{i}c_{i+1} \cdots & \mbox{if $j=\infty$} \\
\varepsilon & \mbox{if $i>j$} \\
\end{array}
\right. 
\end{equation*}
where $\varepsilon$ is {\it{empty word}}, so that 
$$\varepsilon a = a\varepsilon = a$$
for all $a \in \mathbb{Z}$. 
By definition, we see that 
$$\bc[i,j]\bc[j+1,k] = \bc[i,k] \ \mbox{where $i \leq j <k$}.$$
Recall 
$$d_{\beta}^{*}(1) = d_{1}d_{2} \cdots.$$

\begin{Def}
Let $\bc=c_{1}c_{2} \cdots \in D_{\beta}$. 
We say that 
$\{k_{i}\}_{i \geq 0} \subset \mathbb{N}$ gives a free block decomposition 
of $\bc$ if $\{k_{i}\}_{i \geq 0}$ satisfies 
$k_{0}=0$, $k_{i}<k_{i+1}$ and 
\begin{equation*}
c_{k_{i}+j}^{} \left\{ 
\begin{array}{ll}
=d_{j} & \mbox{if $1 \leq j < k_{i+1}-k_{i}$} \\
<d_{j} & \mbox{if $j=k_{i+1}-k_{i}$.} \\
\end{array}
\right. 
\end{equation*}
That is, 
$$c_{1}c_{2} \cdots c_{n} \cdots 
= d_{\beta}^{*}(1)[1,k_{1}-1]c_{k_{1}}d_{\beta}^{*}(1)[1,k_{2}-k_{1}-1]
c_{k_{2}}\cdots.$$
\end{Def}

The notion of a free block decomposition is introduced in \cite{akiyama5}. 
Notice that for each $\bc \in D_{\beta}$, 
there always exists $\{k_{i}\}_{i \geq 0}$ 
which gives a free block decomposition of $\bc$ (note Theorem \ref{parry}). 

For $\bc,\bc' \in \mathbb{N}_{0}^{\mathbb{N}}$, 
define $\overset{\nu}{=}$ by 
$$\bc \overset{\nu}{=} \bc' 
\Longleftrightarrow \nu(\bc) = \nu(\bc')$$
and define the termwise subtraction $\bc \ominus \bc'$ by 
$$\bc \ominus \bc' 
= (c_{1}-c_{1}')(c_{2}-c_{2}')\cdots.$$
We slightly abuse the notation of $\ominus$ as follows: 
$$\bc[i,\infty] \ominus \bc'[j,\infty] 
= (c_{i}-c_{j}')(c_{i+1}-c_{j+1}') \cdots.$$
Clearly, $\nu(\bc \ominus \bc') = \nu(\bc)-\nu(\bc')$. 
Now, let 
$$\bc^{+}_{0}:= \bc[1,\ell-1](c_{\ell}+1)\bc[\ell+1,\infty]$$
such that $\nu(\bc_{0}^{+})<1$ 
where $\bc=c_{1}c_{2} \cdots \in D_{\beta}$ and $\ell \in \mathbb{N}$. 
Then $\bc_{0}^{+}$ does not necessarily belong to $D_{\beta}$. 
When $\bc^{+} \in D_{\beta}$ satisfies $\bc^{+} \overset{\nu}{=} \bc_{0}^{+}$, 
we call $\bc^{+}$ the {\it{normalization}} of $\bc_{0}^{+}$. 
In this subsection, we give a way to transform 
$\bc_{0}^{+}$ into its normalization $\bc^{+}$ 
based on the free block decomposition of $\bc$. 
The following {\it{carry formula}} is used 
in the process of normalization of $\bc_{0}^{+}$. 

\begin{Obs}\label{obs2}
For any $\ell \in \mathbb{N}$, 
$\bc \overset{\nu}{=} 
\bc[1,\ell-1](c_{\ell}+1) 
(\bc[\ell+1,\infty] \ominus d_{\beta}^{*}(1))$. 
\end{Obs}

\begin{proof}
Denote by $\bc'$ the right-hand side. Then 
\begin{equation*}
\begin{split}
\nu(\bc') 
&= \sum_{n=1}^{\ell-1}c_{n}\beta^{-n} + (c_{\ell}+1)\beta^{-\ell} 
+ \beta^{-\ell} (\nu(c_{\ell+1}c_{\ell+2}\cdots) - \nu(d_{\beta}^{*}(1)) ) \\
&= \sum_{n=1}^{\infty}c_{n}\beta^{-n} \\
&= \nu(\bc). \\
\end{split}
\end{equation*}
\end{proof}

The following remark explains how to use the carry formula. 

\begin{Remark}\label{idea}
Suppose that $\bc \in D_{\beta}$ and 
$\ell \in \mathbb{N}$ satisfy $\nu(\bc_{0}^{+})<1$ where 
$$\bc_{0}^{+} := \bc[1,\ell-1](c_{\ell}+1)\bc[\ell+1,\infty].$$
Let $\{k_{j}\}_{j \in \mathbb{N}_{0}}$ 
give the free block decomposition of $\bc$ and 
$k_{i} < \ell \leq k_{i+1}$. 
Then 
\begin{equation*}
\bc_{0}^{+} = \bc[1,k_{i}]d_{\beta}^{*}(1)[1,\ell-k_{i}-1] 
(c_{\ell}+1)\bc[\ell+1,\infty]. 
\end{equation*}
So by Observation \ref{obs2}, 
$$\bc_{0}^{+} \overset{\nu}{=} \bc_{0}^{carry}$$
where 
\begin{equation*}
\begin{array}{l}
\bc_{0}^{carry} := \bc[1,k_{i}-1] 
(c_{k_{i}}+1)0^{\ell-k_{i}-1}\theta \bc'_{1} \\
\theta:=c_{\ell}+1-d_{\ell-k_{i}} \\
\bc_{1}':=\bc[\ell+1,\infty] 
\ominus d_{\beta}^{*}(1)[\ell-k_{i}+1,\infty]. \\
\end{array}
\end{equation*}
Moreover, if $\theta + \nu(\bc_{1}') \in [0,1)$, then we have 
$$\bc_{0}^{+} \overset{\nu}{=} \bc_{1}^{+}$$
where 
$$\bc_{1}^{+} 
:= \bc[1,k_{i}-1] (c_{k_{i}}+1)0^{\ell-k_{i}} 
d_{\beta}(\theta + \nu(\bc'_{1})),$$
because 
$$\beta^{\ell-1}\nu(\bc_{0}^{carry}\ominus \bc_{1}^{+}) 
= \nu(\theta \bc_{1}') 
- \nu(0d_{\beta}(\theta+\nu(\bc_{1}'))) = 0.$$
\end{Remark}

\begin{Example}
Let $\beta>1$ be the algebraic integer with minimal polynomial 
$x^{3}-x^{2}-x-1$ (so-called {\it{tribonacci number}} ). 
Then $d_{\beta}(1) = 111$ and so $d_{\beta}^{*}(1)=(110)^{\infty}$. 
For instance, let 
$$\bc=c_{1}c_{2} \cdots := 10(110)^{2}10^{\infty} \ \mbox{and} \ \ell=9.$$
Then the free block decomposition of $\bc$ is 
$$\{k_{j}\}_{j \in \mathbb{N}} = \{2,10,11,12,13,\cdots\}.$$
By Observation \ref{obs2}, we have 
$$\bc_{0}^{+}:=10(110)^{2}20^{\infty} \overset{\nu}{=} 110^{6}1\bc_{1}' \ \mbox{where} \ \bc_{1}' := 0^{\infty} \ominus d_{\beta}^{*}(1)[8,\infty].$$
So, since $1+\nu(\bc_{1}') = 1-\xi(8) \in [0,1)$ by Lemma \ref{fund_lem}-(2), we get 
$$\bc_{0}^{+} \overset{\nu}{=} \bc_{1}^{+} := 110^{7} d_{\beta}(1+\nu(\bc_{1}')).$$
\end{Example}

Now, for repeated use of the carry formula, we prepare the following lemma. 

\begin{Lem}\label{lem1}
Let $\bc' \in D_{\beta}$ and 
$\bc=c_{1}c_{2} \cdots \in D_{\beta}$ with 
its free block decomposition $\{k_{j}\}_{j \in \mathbb{N}_{0}}$. 
Suppose $i \in \mathbb{N}_{0}$ and $\ell \in \mathbb{N}$ with $\ell >k_{i}$. 
Define 
\begin{equation*}
\tilde{\bc} := \left\{ 
\begin{array}{ll}
\bc[1,\ell-1](c_{\ell}+1)\bc' & \mbox{if $k_{i}< \ell < k_{i+1}$} \\
\bc[1,k_{i+1}-1](c_{k_{i+1}}+1)0^{\ell-k_{i+1}} 
\bc' & \mbox{if $\ell \geq k_{i+1}$} \\
\end{array}
\right. 
\end{equation*}
and 
$$\tilde{\bc}^{carry} 
:= \bc[1,k_{i}-1](c_{k_{i}}+1)0^{\ell-k_{i}-1} \theta 
(\bc' \ominus d_{\beta}^{*}(1) [\ell-k_{i}+1,\infty])$$
where the sequence $0^{0}$ means empty word and 
\begin{equation*}
\theta := \left\{ 
\begin{array}{ll}
1 & \mbox{if $k_{i} < \ell < k_{i+1}$} \\
0 & \mbox{if $\ell \geq k_{i+1}$.} \\
\end{array}
\right. 
\end{equation*}
If $\tilde{\bc} \notin D_{\beta}$, 
then 
\begin{equation}\label{eq:lem1}
\tilde{\bc} \overset{\nu}{=} 
\tilde{\bc}^{carry}
\end{equation}
and 
\begin{equation}\label{eq:lem2}
\theta + \nu(\bc')-\xi(\ell-k_{i}+1) \geq 0. 
\end{equation}
\end{Lem}

The proof is a little complicated, so we defer it for the moment while we exemplify the validity of Lemma \ref{lem1}. 

\begin{Example}
Let $\beta>1$ be the algebraic integer with minimal polynomial 
$x^{3}-x-1$ (so-called {\it{minimal Pisot number}}). Then 
$d_{\beta}(1) = 10^{3}1$ and so $d_{\beta}^{*}(1)=(10^{4})^{\infty}$. 
\begin{description}
\item{(1)} \ Let 
$$\bc := 010^{5} 10^{\infty}, \ \ell=9 \ \mbox{and} \ \bc':=\bc[10,\infty]=0^{\infty}.$$
Then the free block decomposition of $\bc$ is 
$$\{k_{j}\}_{j \in \mathbb{N}} = \{1,7,13,14,\cdots\}.$$
So, since $\tilde{\bc}:=0 10^{5} 110^{\infty} \notin D_{\beta}$, we have by Observation \ref{obs2}, 
$$\tilde{\bc} \overset{\nu}{=} \tilde{\bc}^{carry} := 0 10^{4} 101(\bc' \ominus d_{\beta}^{*}(1)[3,\infty]).$$
Moreover by Lemma \ref{fund_lem}-(2), we have 
$$1+\nu(\bc')-\xi(3) =1-\xi(3) \geq 0.$$
\item{(2)} \ Let
$$\bc := 0 10^{5}10^{\infty}, \ \ell=7 \ \mbox{and} \ \bc':=\bc[8,\infty]=10^{\infty}.$$
Then the free block decomposition of $\bc$ is 
$$\{k_{j}\}_{j \in \mathbb{N}} = \{1,7,13,14,\cdots\}.$$
So, since $\tilde{\bc}:=0 10^{4} 110^{\infty} \notin D_{\beta}$, we have by Observation \ref{obs2}, 
$$\tilde{\bc} \overset{\nu}{=} \tilde{\bc}^{carry} := 10^{6} (\bc' \ominus d_{\beta}^{*}(1)[7,\infty]).$$
Moreover, since $d_{\beta}^{*}(1)[7,\infty] = 0^{4}d_{\beta}^{*}(1)$, we see that 
$$\nu(\bc')-\xi(7) = \beta^{-1}-\beta^{-4} >0.$$
\end{description}
\end{Example}

\begin{proof}[Proof of Lemma \ref{lem1}]
Suppose $k_{i}<\ell <  k_{i+1}$. 
By Remark \ref{idea}, \eqref{eq:lem1} holds. 
Moreover, since $\theta=1$ and $\xi(\ell-k_{i}+1) \leq 1$ by Lemma \ref{fund_lem}-(2), 
we have \eqref{eq:lem2}. 
So consider the case $\ell \geq k_{i+1}$. 
By the assumption, there is $n \in \mathbb{N}$ such that 
\begin{equation}\label{eq:lem3}
\tilde{\bc}[n,\infty] \geq_{lex} d_{\beta}^{*}(1). 
\end{equation}
Since $\tilde{\bc}[k_{i}+1,\infty]=\bc' \in D_{\beta}$, 
we have $n \leq k_{i+1}$. 
First, we show $k_{i}<n\leq k_{i+1}$. 
Assume that $k_{j-1} < n \leq k_{j}$ and $j \leq i$. Then 
by definition of $k_{m}$'s, 
$$\tilde{\bc}[n,k_{j}-1] =d_{\beta}^{*}(1)[n-k_{j-1},k_{j}-k_{j-1}-1] 
\ \mbox{and} \ c_{k_{j}}<d_{k_{j}-k_{j-1}}.$$
So by Lemma \ref{fund_lem}-(2), 
$$\tilde{\bc}[n,\infty] <_{lex}d_{\beta}^{*}(1)[n-k_{j-1},k_{j}-k_{j-1}-1]d_{k_{j}-k_{j-1}}0^{\infty} 
\leq_{lex} d_{\beta}^{*}(1)[n-k_{j-1},\infty] \leq_{lex} d_{\beta}^{*}(1).$$
This contradicts \eqref{eq:lem3} and we get 
$k_{i}<n \leq k_{i+1}$. 
So, since 
$$\tilde{\bc}[k_{i}+1,n-1]\tilde{\bc}[n,\infty] 
\geq_{lex} d_{\beta}^{*}(1)[1,n-k_{i}-1]d_{\beta}^{*}(1),$$
we have by Lemma \ref{fund_lem}-(2), 
\begin{equation}\label{eq:lem4}
\tilde{\bc}[k_{i}+1,\infty] \geq_{lex} d_{\beta}^{*}(1). 
\end{equation}
Note that by definition of $k_{m}$'s, 
\begin{equation*}
\tilde{\bc}[k_{i}+1,\ell] = \bc[k_{i}+1,k_{i+1}-1](c_{k_{i+1}}+1)0^{\ell-k_{i+1}} 
\leq_{lex} d_{\beta}^{*}(1)[1,\ell-k_{i}]. 
\end{equation*}
Therefore by \eqref{eq:lem4}, 
\begin{equation}\label{eq:lem5}
\tilde{\bc}[k_{i}+1,\ell] = d_{\beta}^{*}(1)[1,\ell-k_{i}] 
\end{equation}
So by Observation \ref{obs2}, \eqref{eq:lem1} holds. 
Moreover, since 
$$\tilde{\bc}[\ell+1,\infty] \geq_{lex} d_{\beta}^{*}(1)[\ell-k_{i}+1,\infty]$$
by \eqref{eq:lem4} and \eqref{eq:lem5}, we have by Lemma \ref{fund_lem}-(3), 
$$\nu(\bc')-\xi(\ell-k_{i}+1) 
= \nu(\tilde{\bc}[\ell+1,\infty]) 
- \nu(d_{\beta}^{*}(1)[\ell-k_{i}+1,\infty]) \geq 0.$$
Hence \eqref{eq:lem2} holds. 
\end{proof}

\subsection{Proof of Theorem \ref{key}}

\begin{proof}[Proof of Theorem \ref{key}]
Let $x \in [0,\infty)$ and $\ell:=L(x+1)$ be given. 
Note $\ell \geq L(x)$. 
Now we use {\it{the carry look-ahead}}, 
that is, put
$$\bc=c_{1}c_{2} \cdots:= 0^{\ell-L(x)}d_{\beta}(\beta^{-L(x)}x).$$
Then 
$$d_{\beta}(\beta^{-L(x)}x) = \bc[\ell-L(x)+1,\infty]$$
and so 
\begin{equation*}
\begin{split}
x &= c_{\ell-L(x)+1}\beta^{L(x)-1} 
+ \cdots + c_{\ell-1}\beta + c_{\ell} + c_{\ell+1}\beta^{-1}+\cdots \\
&= c_{1}\beta^{\ell-1} + \cdots + c_{\ell-1}\beta 
+ c_{\ell} + c_{\ell+1}\beta^{-1}+\cdots. \\
\end{split}
\end{equation*}
Define 
$$\bc^{+}_{0}:=\bc[1,\ell-1](c_{\ell}+1)\bc[\ell+1,\infty].$$
Notice that 
$$\beta^{-\ell}(x+1) 
=c_{1}\beta^{-1}+\cdots+c_{\ell-1}\beta^{-\ell+1}+(c_{\ell}+1)\beta^{-\ell}+ 
c_{\ell+1}\beta^{-\ell-1}+\cdots.$$
Thus 
\begin{equation}\label{eq:key-i-0}
\bc_{0}^{+} \overset{\nu}{=} d_{\beta}(\beta^{-\ell}(x+1)).  
\end{equation}
Now we prepare some notations for the normalization of $\bc_{0}^{+}$ based on the free block decomposition $\{k_{j}\}_{j \in \mathbb{N}_{0}}$ of $\bc$. Pick $i$ with $k_{i}<\ell \leq k_{i+1}$. Define 
\begin{equation*}
\theta = \left\{ 
\begin{array}{ll}
1 & \mbox{if $\ell < k_{i+1}$} \\
0 & \mbox{if $\ell=k_{i+1}$.} \\
\end{array}
\right. 
\end{equation*}
Let 
$$\gamma:= \theta+\{x\}_{\beta}-\xi(\ell-k_{i}+1).$$

$Claim$. $\gamma < 1$. 

$(Proof \ of \ Claim)$ It suffices to show the case $\theta=1$. Since $\bc[k_{i}+1,\infty] <_{lex} d_{\beta}^{*}(1)$ 
and $\bc[k_{i}+1,\ell] = d_{\beta}^{*}(1)[1,\ell-k_{i}]$ by definition of $k_{m}$'s, 
we have 
$$\bc[\ell+1,\infty] <_{lex} d_{\beta}^{*}(1)[\ell-k_{i}+1,\infty]$$
and so by Lemma \ref{fund_rem}-(4), 
$$\gamma-1 = \{x\}_{\beta}-\xi(\ell-k_{i}+1) 
= \nu(\bc[\ell+1,\infty])-\nu(d_{\beta}^{*}(1)[\ell-k_{i}+1,\infty]) 
< 0.$$

$(Proof \ Continued)$ 
Let 
$$y_{0}:=\{x\}_{\beta}$$
and in case $\gamma \geq 0$, define $\{y_{n}\}_{1 \leq n \leq i}$ by 
\begin{equation*}
y_{1}:= \gamma \ \mbox{and} \ 
y_{n+1}:=y_{n}-\xi(\ell-k_{i-n}+1) 
\ \mbox{for $n \in \lbc1,i-1\rbc$.} \\
\end{equation*}
Notice that $y_{n}<1$ for all $n \in \lbc0,i\rbc$. 
If $y_{n} \geq 0$, then define 
\begin{equation*}
\bc^{+}_{n}=\bc[1,k_{i-n+1}-1](c_{k_{i-n+1}}+1)0^{\ell-k_{i-n+1}}d_{\beta}(y_{n}). 
\end{equation*}
Recall 
$$\mathscr{X} = 
\left\{ \sum_{m=1}^{p}u_{m}\xi(m) \ \Biggl| \ p \in \mathbb{N}, \ u_{m} \in \mathbb{N}_{0} \ (m \in \lbc1,p\rbc)\right\}.$$
We prove the assertion through $3$ steps as follows: 
\begin{description}
\item{(i)} \ For each $n$, 
the following statement (S$_{n}$) holds: 
\begin{description}
\item{(S$_{n})$} \ If $y_{j} \geq 0$ and 
$\bc_{j}^{+} \notin D_{\beta}$ 
for all $j \in \lbc0,n-1\rbc$, then $y_{n} \geq 0$ and 
$\bc_{n}^{+} \overset{\nu}{=} d_{\beta}(\beta^{-\ell}(x+1))$. 
\end{description}
\item{(ii)} \ There exists $n \in \lbc0,i\rbc$ such that 
$y_{n} \geq 0$ and $\bc_{n}^{+}=d_{\beta}(\beta^{-\ell}(x+1))$. 
\item{(iii)} \ There exists $X \in \mathscr{X}$ such that 
$\{x+1\}_{\beta}-\{x\}_{\beta} =\theta-X$. 
\end{description}

(i) By \eqref{eq:key-i-0}, (S$_{0}$) holds. 
Let $n<i$. Suppose that (S$_{n}$) holds 
and that $y_{j} \geq 0$ and $\bc_{j}^{+} \notin D_{\beta}$ 
for all $j \in \lbc0,n\rbc$. 
By applying Lemma \ref{lem1} with $\tilde{\bc}=\bc_{n}^{+}$ and $\bc'=d_{\beta}(y_{n})$, we have 
$$y_{n+1} = y_{n}-\xi(\ell-k_{i-n}+1) \geq 0$$
and 
\begin{equation*}
\begin{split}
\bc_{n}^{+} &\overset{\nu}{=} \bc[1,k_{i-n}-1](c_{k_{i-n}}+1)0^{\ell-k_{i-n}} (d_{\beta}(y_{n}) \ominus d_{\beta}^{*}(1)[\ell-k_{i-n}+1,\infty]) \\
&\overset{\nu}{=} \bc[1,k_{i-n}-1](c_{k_{i-n}}+1)0^{\ell-k_{i-n}} d_{\beta}(y_{n}-\xi(\ell-k_{i-n}+1)) \\
&\overset{\nu}{=} \bc[1,k_{i-n}-1](c_{k_{i-n}}+1)0^{\ell-k_{i-n}} d_{\beta}(y_{n+1}) \\
&= \bc_{n+1}^{+}. \\
\end{split}
\end{equation*}
Therefore (S$_{n+1}$) holds. Hence by induction on $n$, we have the desired result. 

(ii) Assume that $d_{\beta}(\beta^{-\ell}(x+1)) \neq \bc_{n}^{+}$ 
for all $n \in \lbc0,i\rbc$. 
Then by (i), $y_{j} \geq 0$ 
and $\bc_{j}^{+} \overset{\nu}{=} d_{\beta}(\beta^{-\ell}(x+1))$ 
for all $j$. So 
$$\nu(\bc_{i}^{+}) = \beta^{-\ell}(x+1) < 1.$$
On the other hand, 
since $\bc_{i}^{+} \notin D_{\beta}$, 
there exists $m \in \mathbb{N}$ such that 
$$\bc_{i}^{+}[m,\infty] \geq_{lex} d_{\beta}^{*}(1).$$
Note $m \in \lbc1,k_{1}\rbc$ because $\bc_{i}^{+}[k_{1}+1,\infty]= 0^{\ell-k_{1}}d_{\beta}(y_{i}) \in D_{\beta}$. By Lemma \ref{fund_lem}-(2), 
\begin{equation*}
\begin{split}
\bc_{i}^{+} = d_{\beta}^{*}(1)[1,m-1]\bc_{i}^{+}[m,\infty] 
\geq_{lex} d_{\beta}^{*}(1)[1,m-1]d_{\beta}^{*}(1) 
\geq_{lex} d_{\beta}^{*}(1). 
\end{split}
\end{equation*}
So by Lemma \ref{fund_lem}-(3), 
$$\nu(\bc^{+}_{i}) \geq \nu(d_{\beta}^{*}(1))=1.$$
Hence we have a contradiction. 

(iii) Suppose that $\bc_{n}^{+}=d_{\beta}(\beta^{-\ell}(x+1))$. 
Define 
\begin{equation*}
X= \left\{ 
\begin{array}{ll}
0 & \mbox{if $n=0$} \\
\sum_{j=0}^{n-1} \xi(\ell-k_{i-j}+1) & \mbox{if $n \geq 1$.} \\
\end{array}
\right. 
\end{equation*}
If $n=0$, then $\ell=k_{i+1}$ and so $\theta=0$. Therefore in this case, $\{x+1\}_{\beta}-\{x\}_{\beta}=0=\theta-X$. Consider the case $n \geq 1$. 
By definition of $\{y_{n}\}_{n \geq 1}$, 
\begin{equation*}
y_{n} = y_{n-1}-\xi(\ell-k_{i-n+1}+1) 
= \cdots = \gamma -\sum_{j=1}^{n-1} \xi(\ell-k_{i-j}+1) 
= \theta+\{x\}_{\beta}-X. 
\end{equation*}
So, since $d_{\beta}(\{x+1\}_{\beta}) = d_{\beta}(\beta^{-\ell}(x+1))[\ell+1,\infty]$ by \eqref{eq:key-i-0}, we have 
\begin{equation*}
\{x+1\}_{\beta}= y_{n} = \theta+\{x\}_{\beta}-X. 
\end{equation*}
Thus (iii) holds. 

Hence by (iii) and Remark \ref{fund_rem}, 
the proof is completed. 
\end{proof}

As a corollary of Theorem \ref{key}, we have the following. 

\begin{Cor}\label{cor1}
$$\{ \{N\}_{\beta} \ | \ N \in \mathbb{N}_{0} \} 
\subset \left\{ \left\{ - \sum_{n=1}^{q} \omega_{n}T^{n}(1) \right\} \ \Biggl| \ q \in \mathbb{N}, \ \omega_{n} \in \mathbb{N}_{0} \ (n \in \lbc1,q\rbc) \right\}.$$
\end{Cor}

\begin{proof}
It suffices to show that for $N \in \mathbb{N}$, 
there exist $q \in \mathbb{N}$ and $\omega_{n} \in \mathbb{N}_{0}$ 
$(n \in \lbc1,q\rbc)$ such that 
\begin{equation}\label{eq:cor1}
\{N\}_{\beta} \equiv -\sum_{n=1}^{q}\omega_{n}T^{n}(1) \mod{\mathbb{Z}}
\end{equation}
by induction on $N$. Suppose that (\ref{eq:cor1}) holds for $N$. 
By Theorem \ref{key}, there exist $\theta \in \{0,1\}$, 
$q' \in \mathbb{N}$ 
and $\omega_{n}' \in \mathbb{N}_{0}$ $(n \in \lbc0,q'\rbc)$ such that 
$$\{N+1\}_{\beta}-\{N\}_{\beta} 
= \theta-\sum_{n=0}^{q'}\omega_{n}'T^{n}(1).$$
Let $q_{0}:=\max\{q,q'\}$. Define 
$\omega_{n}:=0$ for $n \in \lbc q+1,q_{0}\rbc$ and 
$\omega_{n}':=0$ for $n \in \lbc q'+1,q_{0}\rbc$. 
Then we have 
\begin{equation*}
\{N+1\}_{\beta} 
= \theta-\sum_{n=0}^{q_{0}}\omega_{n}'T^{n}(1) + \{N\}_{\beta} 
\equiv -\sum_{n=1}^{q_{0}}(\omega_{n}'+\omega_{n})T^{n}(1) \mod{\mathbb{Z}}. 
\end{equation*}
Hence we have the desired result. 
\end{proof}

Notice that if $\{N\}_{\beta}$ has the finite beta-expansion for each $N\in\mathbb{N}$, then $\beta$ has property (F$_{1}$). 
So by Corollary \ref{cor1}, we see that if $\{ - \sum_{n=1}^{q} \omega_{n}T^{n}(1) \} \in \mbox{Fin}(\beta)$ for all $q \in \mathbb{N}$ and $\omega_{n} \in \mathbb{N}_{0}$ $(n \in \lbc1,q\rbc)$, then $\beta$ has property (F$_{1}$). 

\section{Sufficient Condition for Property (F$_{1}$)}\label{sec:main2}

In this section, let $\beta>1$ be an algebraic integer with minimal polynomial 
$$x^{d}-a_{d-1}x^{d-1}-\cdots -a_{1}x-a_{0}$$
and 
$$\br = (r_{1},r_{2},\cdots,r_{d-1}) := \left(a_{0}\beta^{-1}, \cdots,\sum_{i=1}^{j}a_{j-i}\beta^{-i}, \cdots, \sum_{i=1}^{d-1}a_{d-1-i}\beta^{-i} \right).$$

\subsection{Shift radix system}

For $\wl=(l_{1},l_{2},\cdots,l_{d-1}) \in \mathbb{Z}^{d-1}$, let  
$$\lambda(\wl)= \br \cdot \wl$$
where $\br \cdot \wl$ is the inner product of the vector $\br$ and $\wl$. 
Define $\tau:\mathbb{Z}^{d-1} \to \mathbb{Z}^{d-1}$ by 
$$\tau(\wl) =  (l_{2},\cdots,l_{d-1},-\lfloor \lambda(\wl) \rfloor).$$
In \cite{akiyama6}, $\tau$ is 
called {\it{shift radix system}} (SRS for short). 
The definition of SRS is different from previous articles (indeed, the SRS in \cite{akiyama3,akiyama4} coincide with SRS satisfying the property that for each $\wl \in \mathbb{Z}^{d-1}$, there is $k \in \mathbb{N}$ such that $\tau^{k}(\wl)=\bzero$). 
In this paper, the definition of SRS conforms to \cite{akiyama6}.  
Let 
$$\{\lambda\}(\wl) = \{\lambda(\wl)\}.$$
Then $\{\lambda\}$ is bijective and we have the following commutative diagram 
\begin{equation*}
\begin{array}{ccc}
\mathbb{Z}^{d-1} & \overset{\tau}{\longrightarrow} & \mathbb{Z}^{d-1} \\
\{\lambda\} \Big\downarrow & & \Big\downarrow \{\lambda\} \\
\mathbb{Z}[\beta] \cap [0,1) & \underset{T}{\longrightarrow} & \mathbb{Z}[\beta]\cap [0,1) \\
\end{array}
\end{equation*}
(for example, see Proposition 2.4 in \cite{survey}). For the completeness' sake, we prove this commutative diagram. Notice that 
\begin{equation}\label{eq:prop1}
\beta-a_{d-1} 
= a_{d-2}\beta^{-1}+\cdots+a_{1}\beta^{-d+2}+a_{0}\beta^{-d+1} 
= \lambda(0,\cdots,0,1). 
\end{equation}
Let $\ba=(a_{0},a_{1},\cdots,a_{d-2})$. Then for $\wl=(l_{1},l_{2},\cdots,l_{d-1})$, 
\begin{equation*}
\beta \lambda(\wl) = \ba \cdot \wl + r_{1}l_{2} + \cdots + r_{d-2}l_{d-1} = \ba \cdot \wl + \lambda(\tau(\wl)) + \lfloor \lambda(\wl) \rfloor(\beta-a_{d-1}) 
\end{equation*}
Therefore $\beta \{\lambda\} (\wl) = \ba \cdot \wl + \lambda(\tau(\wl)) - \lfloor \lambda(\wl)\rfloor a_{d-1}$ and so $\beta \{\lambda\}(\wl) \equiv \lambda(\tau(\wl)) \mod{\mathbb{Z}}$. 
Hence 
$$T(\{\lambda\}(\wl)) = \{\lambda\}(\tau(\wl)).$$

Define 
$$F_{\beta} := \{ \wl \in \mathbb{Z}^{d-1} \ 
| \ \exists k\geq 0; \ \tau^{k}(\wl)=\bzero\}.$$
Since $\{\lambda\}(\bzero) = 0$, 
we have 
\begin{equation}\label{fin_beta}
\{\lambda\}(F_{\beta}) \subset \mbox{Fin}(\beta) 
\end{equation}
by commutative diagram. Let 
$$\wl_{I}:=(0,\cdots,0,1) \in \mathbb{Z}^{d-1}.$$
Then we have the following. 

\begin{Remark}\label{T1_exp}
$T^{n}(1) 
= \{\lambda\}(\tau^{n-1}(\wl_{I}))$ 
for all $n \in \mathbb{N}$. 
\end{Remark}

\begin{proof}
By \eqref{eq:prop1}, 
$$T(1) = \beta-\lfloor \beta \rfloor 
= \beta-a_{d-1}-\lfloor \beta-a_{d-1}\rfloor 
= \{\lambda\}(\wl_{I}).$$
Hence by commutative diagram, we have the desired result. 
\end{proof}

Now, let 
$$V_{\beta}:=\left\{ -\sum_{n=1}^{r}\omega_{n}\tau^{n-1}(\wl_{I}) \ \Biggl| \ r \in \mathbb{N}, \ \omega_{n} \in \mathbb{N}_{0} \ (n \in \lbc1,r\rbc) \right\}.$$

\begin{Prop}\label{prop}
If $V_{\beta} \subset F_{\beta}$, then $\beta$ has property (F$_{1}$). 
\end{Prop}

\begin{proof}
By Corollary \ref{cor1}, it suffices to show 
$\left\{ -\sum_{n=1}^{r} \omega_{n}T^{n}(1) \right\} \in \mbox{Fin}(\beta)$ for all $r \in \mathbb{N}$ and $\omega_{n} \in \mathbb{N}_{0}$ $(n \in \lbc1,r\rbc)$. 
Notice that by linearity of $\lambda$, 
\begin{equation*}
\begin{split}
\{\lambda\} ( \zeta_{1} \wl  + \zeta_{2} \wk ) 
&\equiv \lambda(\zeta_{1} \wl + \zeta_{2} \wk) \mod{\mathbb{Z}} \\
&= \zeta_{1}\lambda(\wl) + \zeta_{2} \lambda(\wk) \\
&\equiv \zeta_{1} \{\lambda\}(\wl) + \zeta_{2} \{\lambda\}(\wk) \mod{\mathbb{Z}} \\
\end{split}
\end{equation*}
for all $\zeta_{1},\zeta_{2} \in \mathbb{Z}$ and $\wl,\wk \in \mathbb{Z}^{d-1}$. So by Remark \ref{T1_exp}, 
\begin{equation*}
\begin{split}
\{\lambda\} 
\left( -\sum_{n=1}^{r}\omega_{n}\tau^{n-1}(\wl_{I}) \right) 
&\equiv -\sum_{n=1}^{r}\omega_{n}\{\lambda\}(\tau^{n-1}(\wl_{I})) \mod{\mathbb{Z}} \\
&= -\sum_{n=1}^{r} \omega_{n}T^{n}(1) \\
\end{split}
\end{equation*}
Therefore, since the left-hand side belongs to $[0,1)$, we have 
$$\{\lambda\} \left( -\sum_{n=1}^{r}\omega_{n}\tau^{n-1}(\wl_{I}) \right) = \left\{ -\sum_{n=1}^{r} \omega_{n}T^{n}(1) \right\}.$$
Hence by \eqref{fin_beta}, we get the assertion. 
\end{proof}

Let 
\begin{equation*}
\begin{split}
&Q_{\beta} = \left\{ \wl \in \mathbb{Z}^{d-1} \ \left| \ 
\begin{split}
&\exists N \in \mathbb{N}, \ \exists \{\wl_{n}\}_{n=1}^{N} \ \mbox{s.t.} \ 
\wl_{1}=\wl_{I}, \ \wl_{N}=\wl \\
&\mbox{and} \ \wl_{n+1} \in \{\tau(\wl_{n}), 
\tau^{*}(\wl_{n})\} \ (\forall n) \\
\end{split}
\right. 
\right\} \\
&P_{\beta}=\{\wl \in Q_{\beta}\setminus\{\bzero\} \ | \ 
\exists N \in \mathbb{N}; \ \tau^{N}(\wl)=\wl\}. \\
\end{split}
\end{equation*}
where 
$$\tau^{*}(\wl) = -\tau(-\wl).$$
It is known that if $\beta$ is a Pisot number, 
then $\tau$ and $\tau^{*}$ is contractive 
(see \cite{akiyama3,survey}). 
In particular, $Q_{\beta}$ is finite (cf. \cite{mine1}). 

The purpose in this section is to prove the following theorem. 

\begin{Thm}\label{main2}
Suppose that a Pisot number $\beta$ satisfies the followings: 
\begin{description}
\item{(i)} \ $\tau^{-1}(P_{\beta}) \subset P_{\beta}$. 
\item{(ii)} \ $\lbc-\delta,\delta \mbox{\textnormal{$\rbc$}}^{d-1} 
\cap V_{\beta} \subset F_{\beta}$ where 
$$\delta:=\max\{ |l_{j}| \ | \ (l_{1},\cdots,l_{d-1}) \in P_{\beta}, j\in \lbc1,d-1\mbox{\textnormal{$\rbc$}}\}.$$
\end{description}
Then $\beta$ has property (F$_{1}$). 
\end{Thm}

\subsection{Preparation for the Proof of Theorem \ref{main2}}

Let 
$$-Q_{\beta}:= \{-\wl \in \mathbb{Z}^{d-1} \ | \ \wl \in Q_{\beta}\}.$$

\begin{Obs}\label{obs3}
If $-\wl_{I} \in Q_{\beta}$, 
then $Q_{\beta}=-Q_{\beta}$. 
\end{Obs}

\begin{proof}
It suffices to show $-Q_{\beta} \subset Q_{\beta}$. 
Let $\wl \in -Q_{\beta}$. Then there is $\{\wk_{n}\}_{n=1}^{N}$ such that 
$\wk_{1}=-\wl_{I}$, $\wk_{N}=\wl$ and 
$\wk_{n+1} \in \{\tau(\wk_{n}),\tau^{*}_{\beta}(\wk_{n})\}$ 
because $\tau(\wk)=-\tau^{*}(-\wk)$. 
Since $-\wl_{I} \in Q_{\beta}$, 
we have $\wk_{n} \in Q_{\beta}$ for all $n \in \lbc1,N\rbc$. 
Hence $\wl=\wk_{N} \in Q_{\beta}$. 
\end{proof}

In this subsection, we aim to prove the following proposition. 

\begin{Prop}\label{p_beta}
Let $\beta$ be a Pisot number. 
If $\tau^{-1}(P_{\beta}) \subset P_{\beta}$, 
then $Q_{\beta}=-Q_{\beta}$. 
\end{Prop}

The following lemma states fundamental properties of SRS $\tau$. 

\begin{Lem}\label{tau}
 \ 
\begin{description}
\item{(1)} \ If $\wl \neq \bzero$, 
then $\tau^{*}(\wl) = \tau(\wl)-\wl_{I}$. 
\item{(2)} \ If $(Q_{\beta}\setminus \{\bzero\}) \cap F_{\beta} \neq \emptyset$, 
then $-Q_{\beta} = Q_{\beta}$. 
\item{(3)} \ Let $\wl,\wl' \in \mathbb{Z}^{d-1}$ be given. Then 
\begin{description}
\item{(i)} \ 
$\tau(\wl+\wl') \in 
\{ \tau(\wl)+\tau(\wl'), 
\tau(\wl)+\tau^{*}_{\beta}(\wl')\}$. 
\item{(ii)} \ $\tau(\wl-\wl') \in 
\{ \tau(\wl)-\tau(\wl'), 
\tau(\wl)-\tau^{*}_{\beta}(\wl')\}$. 
\end{description}
\item{(4)} \ Let $\beta$ be a Pisot number. If $\tau^{-1}(P_{\beta}) \subset P_{\beta}$, then $Q_{\beta} \setminus F_{\beta} \subset P_{\beta}$. 
\end{description}
\end{Lem}

When $\beta$ is a Pisot number, it is not known whether the converse of Lemma \ref{tau}-(4) holds. 

\begin{proof}[Proof of Lemma \ref{tau}]
(1) Let $\wl=(l_{1},l_{2},\cdots,l_{d-1}) \in \mathbb{Z}^{d-1}$ 
and $\wl \neq \bzero$. 
By the bijectivity of $\{\lambda\}$, 
$\{\lambda(\wl)\} \neq 0$ and so $\lambda(\wl) \notin \mathbb{Z}$. 
Therefore $-\lfloor \lambda(\wl) \rfloor - 1 
< \lambda(-\wl) < -\lfloor \lambda(\wl) \rfloor$. 
Hence 
\begin{equation*}
\begin{split}
\tau^{*}(\wl) 
&= -(-l_{2},\cdots,-l_{d-1},-\lfloor \lambda(-\wl) \rfloor) \\
&= (l_{2},\cdots,l_{d-1},-\lfloor \lambda(\wl)\rfloor-1) \\
&= \tau(\wl) - \wl_{I}. \\
\end{split}
\end{equation*}

(2) By Observation \ref{obs3}, 
it suffices to show $-\wl_{I} \in Q_{\beta}$. 
Let $\wl \in (Q_{\beta}\setminus \{\bzero\}) \cap F_{\beta}$ be given. 
Then $\tau^{N-1}(\wl) \neq \bzero$ and $\tau^{N}(\wl) = \bzero$ 
for some $N \in \mathbb{N}$. 
So by definition of $Q_{\beta}$ and (1), 
$$Q_{\beta} \ni \tau^{*}(\tau^{N-1}(\wl)) 
= \tau^{N}(\wl)-\wl_{I} 
= -\wl_{I}.$$

(3) By definition of $\tau^{*}$, 
\begin{equation*}
\begin{split}
&\tau(\wl)-\tau^{*}_{\beta}(\wl') 
= \tau(\wl)+\tau(-\wl') \\
&\tau(\wl)-\tau(\wl') 
= \tau(\wl)+\tau^{*}(-\wl'). \\
\end{split}
\end{equation*}
Thus (i) implies (ii). So it suffices to show (i). 
By definition of $\tau$ and $\tau^{*}$, 
it suffices to consider only $(d-1)$-th coordinates 
of $\tau(\wl+\wl')$, 
$\tau(\wl)+\tau(\wl')$ and 
$\tau(\wl)+\tau^{*}(\wl')$. 
Since 
$$\lfloor \lambda(\wl) \rfloor 
+ \lfloor \lambda(\wl') \rfloor 
\leq \lambda(\wl+\wl') 
< \lfloor \lambda(\wl) \rfloor 
+ \lfloor \lambda(\wl') \rfloor +2,$$
we have 
$$\lfloor \lambda(\wl) \rfloor 
+ \lfloor \lambda(\wl') \rfloor 
\leq \lfloor \lambda(\wl+\wl') \rfloor 
\leq \lfloor \lambda(\wl) \rfloor 
+ \lfloor \lambda(\wl') \rfloor +1$$
and so 
$$\tau(\wl+\wl') \in 
\{\tau(\wl)+\tau(\wl'), 
\tau(\wl)+\tau(\wl')-\wl_{I}\}.$$
Hence (3) holds by (1). 

(4) Let $\wl \in Q_{\beta} \setminus F_{\beta}$ be given. 
Since $Q_{\beta}$ is finite, 
$\tau^{N}(\wl) \in P_{\beta}$ for some $N \in \mathbb{N}_{0}$. 
Hence (4) holds. 
\end{proof}

\begin{proof}[Proof of Proposition \ref{p_beta}]
If $(Q_{\beta}\setminus \{\bzero\}) \cap F_{\beta} \neq \emptyset$, 
then $Q_{\beta}=-Q_{\beta}$ by Lemma \ref{tau}-(2). 
Consider the case 
$(Q_{\beta}\setminus \{\bzero\}) \cap F_{\beta}=\emptyset$. 
By Observation \ref{obs3}, 
it suffices to show $-\wl_{I} \in Q_{\beta}$. 
Note $\wl_{I} \notin F_{\beta}$ in this case. 
So by Lemma \ref{tau}-(4), $\wl_{I} \in P_{\beta}$. 
Since $\tau^{-1}(P_{\beta}) \subset P_{\beta}$, 
there is $(\eta,0,\cdots,0) \in P_{\beta}$ such that 
$\tau(\eta,0,\cdots,0)=\wl_{I}$. 
Similarly, $\tau(\zeta,\eta,0,\cdots,0)=(\eta,0,\cdots,0)$ 
for some $(\zeta,\eta,0,\cdots,0) \in P_{\beta}$. 
So, since $(\zeta,\eta,0,\cdots,0) \neq \bzero$, 
we have by Lemma \ref{tau}-(1), 
$$Q_{\beta} \ni \tau^{*}(\zeta,\eta,0,\cdots,0) 
= \tau(\zeta,\eta,0,\cdots,0)-\wl_{I} 
= (\eta,0,\cdots,0,-1).$$
Notice that by Lemma \ref{tau}-(3)-(i), 
\begin{equation*}
\begin{split}
\tau(\eta,0,\cdots,0,-1) 
&= \tau((\eta,0,\cdots,0) - \wl_{I}) \\
&\in \{ \tau(\eta,0,\cdots,0)+\tau(-\wl_{I}), 
\tau^{*}(\eta,0,\cdots,0)+\tau(-\wl_{I})\}. \\
\end{split}
\end{equation*}
Now we prove the assertion by classifying the following cases: 
\begin{description}
\item{$Case \ 1$}. $\tau(\eta,0,\cdots,0,-1) 
= \tau(\eta,0,\cdots,0)+\tau(-\wl_{I})$. 
\item{$Case \ 2$}. $\tau(\eta,0,\cdots,0,-1) 
= \tau^{*}(\eta,0,\cdots,0)+\tau(-\wl_{I})$. 
\end{description}

$Case \ 1$: By Lemma \ref{tau}-(1), 
\begin{equation*}
\begin{split}
Q_{\beta} \ni \tau^{*}(\eta,0,\cdots,0,-1) 
&= \tau(\eta,0,\cdots,0,-1)-\wl_{I} \\
&= \tau(\eta,0,\cdots,0)+\tau(-\wl_{I})-\wl_{I} \\
&= \tau(-\wl_{I}). \\
\end{split}
\end{equation*}
Since $(Q_{\beta} \setminus\{\bzero\}) \cap F_{\beta} = \emptyset$, 
$\tau(-\wl_{I}) \notin F_{\beta}$ 
and so $\tau(-\wl_{I}) \in P_{\beta}$ 
by Lemma \ref{tau}-(4). 
Hence $-\wl_{I} \in Q_{\beta}$ because 
$\tau^{-1}(P_{\beta}) \subset P_{\beta}$. 

$Case \ 2$: By Lemma \ref{tau}-(1), 
\begin{equation*}
\begin{split}
Q_{\beta} \ni \tau(\eta,0,\cdots,0,-1) 
&= \tau^{*}(\eta,0,\cdots,0)+\tau(-\wl_{I}) \\
&= \tau(\eta,0,\cdots,0)-\wl_{I}+\tau(-\wl_{I}) \\
&= \tau(-\wl_{I}). \\
\end{split}
\end{equation*}
So in the same way as Case 1, 
we have the desired result. 
\end{proof}

\subsection{Proof of Theorem \ref{main2}}

Recall 
$$V_{\beta} = 
\left\{ -\sum_{n=1}^{r} \omega_{n}\tau^{n-1}(\wl_{I}) \ \Biggl| \ r \in \mathbb{N}, \ \omega_{n} \in \mathbb{N}_{0} \ (n \in \lbc1,r\rbc) \right\}$$
and 
$$\delta=\max\{ |l_{j}| \ | \ (l_{1},\cdots,l_{d-1}) \in P_{\beta}, j\in \lbc1,d-1\mbox{\textnormal{$\rbc$}}\}.$$
Define
\begin{equation*}
\begin{array}{l}
R_{0}=\lbc-\delta,\delta \rbc^{d-1} \cap V_{\beta} \\
R_{n}= 
\{ \wl \in V_{\beta} \ | \ \wl=\wl'-\tau^{j}(\wl_{I}) \ 
\mbox{for some $j \in \mathbb{N}_{0}$ and $\wl' \in R_{n-1}$} \} \ (n \in \mathbb{N}) \\
R=\bigcup_{n=1}^{\infty} R_{n}. 
\end{array}
\end{equation*} 

\begin{Remark}\label{V}
$V_{\beta} = R$ holds. 
\end{Remark}

\begin{proof}
Clearly, $R \subset V_{\beta}$. So it suffices to show $V_{\beta} \subset R$. 
Let $\wl \in V_{\beta}$ be given. 
Since $\bzero \in R_{0} \subset R$, 
we may assume $\wl \neq \bzero$. By definition of $V_{\beta}$, 
there exist $\{n_{j}\}_{j=1}^{p} \subset \mathbb{N}$ and 
$\omega_{n_{j}} \in \mathbb{N}$ such that 
$$\wl=-\sum_{j=1}^{p}\omega_{n_{j}}\tau^{n_{j}-1}(\wl_{I}).$$
Let $v_{q} := \sum_{j \in \lbc1,q\rbc} \omega_{n_{j}}$ ($q \in \lbc0,p\rbc$). 
For $m \in \lbc1,v_{p}\rbc$, define 
\begin{equation*}
\wl_{m} := - \sum_{j=1}^{q-1} \omega_{n_{j}}\tau^{n_{j}-1}(\wl_{I}) 
- (m-v_{q-1})\tau^{n_{q}-1}(\wl_{I}) \ 
\mbox{if $m \in \lbc v_{q-1}+1,v_{q}\rbc$.} 
\end{equation*}
Now we want to show $\wl_{m} \in R$ for all $m \in \lbc1,v_{p}\rbc$. 
Since $\bzero \in R_{0}$, 
we have $\wl_{1}=-\tau^{n_{1}-1}(\wl_{I}) \in R_{1}$. 
Suppose that $v_{q-1}<m \leq v_{q}$ and $\wl_{m} \in R$. 
Thus $\wl_{m} \in R_{k}$ for some $k \in \mathbb{N}_{0}$. 
Since 
\begin{equation*}
\wl_{m+1} = \left\{ 
\begin{array}{ll}
\wl_{m}-\tau^{n_{q}-1}(\wl_{I}) & \mbox{if $m<v_{q}$} \\
\wl_{m}-\tau^{n_{q+1}-1}(\wl_{I}) 
& \mbox{if $m=v_{q}$,}
\end{array}
\right. 
\end{equation*}
we have $\wl_{m+1} \in R_{k+1}$. 
Hence by induction on $m$, we get the assertion. 
\end{proof}

\begin{proof}[Proof of Theorem \ref{main2}]
By Proposition \ref{prop} and Remark \ref{V}, 
it suffices to show that $R \subset F_{\beta}$. 
Assume that $R \setminus F_{\beta} \neq \emptyset$. 
Then we can define 
$$n=\min\{ j \in \mathbb{N}_{0} \ | \ R_{j} \setminus F_{\beta} \neq \emptyset\}.$$
Let $\wl \in R_{n} \setminus F_{\beta}$. 
By definition of $R_{n}$, 
we can pick $j \in \mathbb{N}_{0}$ and $\wl' \in R_{n-1}$ such that 
$\wl= \wl'-\tau^{j}(\wl_{I})$. 
Define 
$$\wk_{1}=\tau^{j}(\wl_{I}) \ \mbox{and} \ 
\wk_{m+1}= \tau^{m}(\wl') 
- \tau(\tau^{m-1}(\wl')-\wk_{m}) \ (m \geq 1).$$
Then by Lemma \ref{tau}-(3)-(ii), 
$\wk_{m+1} \in \{\tau(\wk_{m}), \tau^{*}(\wk_{m})\}$ 
for $m \geq 1$. 
So $\wk_{m} \in Q_{\beta}$ for all $m \in \mathbb{N}$ 
because $\wk_{1} \in Q_{\beta}$. 
On the other hand, 
since $R_{n-1} \subset F_{\beta}$ by definition of $n$, 
there is $N \in \mathbb{N}_{0}$ such that 
$\tau^{N}(\wl')=\bzero$. Then we get 
$$\tau^{N}(\wl) 
= \tau^{N-1}(\tau(\wl')-\wk_{2}) 
= \cdots = \tau^{N}(\wl')-\wk_{N+1} 
= -\wk_{N+1} \in -Q_{\beta}.$$
Since $\tau^{N}(\wl)=-\wk_{N+1} \in Q_{\beta}$ 
by Proposition \ref{p_beta} and 
$\tau^{N}(\wl) \notin F_{\beta}$ by the assumption of $\wl$, 
we have $\tau^{N}(\wl) \in P_{\beta}$ by Lemma \ref{tau}-(4). 
Therefore $\wl \in P_{\beta}$ 
because $\tau^{-1}(P_{\beta}) \subset P_{\beta}$. 
So 
$$\wl \in \lbc-\delta,\delta\rbc^{d-1} \cap V_{\beta} = R_{0} \subset F_{\beta}.$$
This contradicts $\wl \notin F_{\beta}$. 
\end{proof}

\subsection{The expansion of $\lfloor \beta \rfloor+1$}

The following remark is used in subsection \ref{subsec:CPcase}. 

\begin{Remark}\label{floor_beta+1}
Suppose $\beta\geq (1+\sqrt{5})/2$. Then 
$\lfloor \beta\rfloor +1$ has the finite beta-expansion 
if and only if $-\wl_{I} \in F_{\beta}$. 
\end{Remark}

\begin{proof}
Since $\beta\geq (1+\sqrt{5})/2$, we have $\beta+1 \leq \beta^{2}$ and so 
$$\lfloor \beta \rfloor +1<\beta^{2}.$$
That is, $L(\lfloor \beta \rfloor +1) = 2$. Now we have 
$$T((\lfloor \beta \rfloor +1 )\beta^{-2}) 
= \beta^{-1}(\lfloor \beta \rfloor + 1)-1 
= \beta^{-1}(\lfloor \beta \rfloor + 1-\beta).$$
Therefore, since $\lfloor \beta \rfloor + 1 -\beta=\{-\beta\}$, 
we get 
$$\mbox{the beta-expansion of $\lfloor \beta \rfloor+1$} 
= 10.d_{\beta}(\{-\beta\}).$$
Thus it suffices to show that 
$\{\lambda\}(\tau^{n}(-\wl_{I})) 
= T^{n}(\{-\beta\})$ for all $n \in \mathbb{N}_{0}$. 
Recall $\lambda(\wl_{I}) = \beta-a_{d-1}$ by \eqref{eq:prop1}. 
So, since $\lambda(-\wl_{I}) = a_{d-1}-\beta$, we have 
$$\{\lambda\}(-\wl_{I}) \equiv a_{d-1}-\beta 
\equiv -\beta \mod{\mathbb{Z}}.$$
Hence by the commutative diagram, 
we have the desired result. 
\end{proof}

\section{Cubic Pisot case}\label{sec:main1}

First, we introduce some important properties of Pisot numbers. 
By relations between roots and coefficients, we have the following observation. 

\begin{Obs}\label{c_ineq}
Let $\beta>1$ be a Pisot number with minimal polynomial $x^{3}-ax^{2}-bx-c$. 
Then $|c|< \beta$. 
\end{Obs}

The following results have been shown by Akiyama.  

\begin{Lem}[(\cite{akiyama1})]\label{cp}
Let $\beta>1$ be an cubic number $\beta$ with minimal polynomial 
$x^{3}-ax^{2}-bx-c$. Then $\beta$ is a Pisot number if and only if 
$|b-1|<a+c$ and $c^{2}-b<\mbox{{\rm{sgn}}}(c)(1+ac)$ hold. 
\end{Lem}

\begin{Thm}[(\cite{akiyama2})]\label{pf}
Let $\beta>1$ be a real number with property (PF). 
Then either $\beta$ satisfies property (F) or $\beta$ is 
a Pisot number whose minimal polynomial is of the form: 
$$x^{d}-(\lfloor \beta \rfloor+1)x^{d-1}+\sum_{i=2}^{d}a_{i}x^{d-i}$$
with $a_{i}\geq 0$ $(i=2,\cdots,d)$, $a_{d}>0$ and 
$\sum_{i=2}^{d}a_{i}<\lfloor \beta \rfloor$. Conversely, 
if $\beta>1$ is a root of the polynomial 
$$x^{d}-Bx^{d-1}+\sum_{i=2}^{d}a_{i}x^{d-i}$$
with $a_{i} \geq 0$ $(i=2,\cdots,d)$, $a_{d}>0$ and 
$B>1+\sum_{i=2}^{d}a_{i}$, then $\beta$ is a Pisot number 
with property (PF) and without (F). 
\end{Thm}

\subsection{Proof of Theorem \ref{main1}}\label{subsec:main1}

In this subsection, let $\beta>1$ be the algebraic integer with minimal polynomial 
$$x^{3}-2tx^{2}+2tx-t \ (t \in \mathbb{N} \cap [2,\infty)).$$

\begin{Remark}\label{rem:4(-4)2}
$\beta$ is a Pisot number and does not have property (PF). 
\end{Remark}

\begin{proof}
By Lemma \ref{cp} and Theorem \ref{pf}. 
\end{proof}

So in order to prove Theorem \ref{main1}, it suffices to show that the above $\beta$ satisfies conditions of Theorem \ref{main2}. 

\begin{Obs}\label{obs:4(-4)2}
 \ 
\begin{multicols}{3}
\begin{description}
\item{(1)} \ $\beta<2t-1$. 
\item{(2)} \ $\beta>2t-2$. 
\end{description}
\end{multicols}
\end{Obs}

\begin{proof}
Notice that 
\begin{equation}\label{eq:obs-2}
1=2t\beta^{-1}-2t\beta^{-2}+t\beta^{-3}. 
\end{equation}
Therefore we have 
\begin{equation*}
\beta =2t-2t\beta^{-1}+t\beta^{-2} 
= 2t-1-t\beta^{-2}+t\beta^{-3} =: h. 
\end{equation*}
So 
$$h = 2t-1-t\beta^{-3}(\beta-1)<2t-1-t\beta^{-3} < 2t-1.$$
Thus (1) holds. On the other hand, 
we have by \eqref{eq:obs-2} and Observation \ref{c_ineq}, 
\begin{equation*}
\begin{split}
h &= 2t-2 + 2t\beta^{-1} - 3t\beta^{-2}+2t\beta^{-3} \\
&= 2t-2 + (2t-1)\beta^{-1}-t\beta^{-2}+t\beta^{-4} \\
&> 2t-2 + (2t-2)\beta^{-1}+t\beta^{-4} > 2t-2 \\
\end{split}
\end{equation*}
and hence (2) holds. 
\end{proof}

Now, note that 
$$\lambda(l_{1},l_{2}) 
= t(l_{1}-2l_{2})\beta^{-1}+tl_{2}\beta^{-2}.$$

\begin{Remark}\label{lambda}
 \ 
\begin{description}
\item{(1)} \ $0<\lambda(2,1)<\lambda(1,0) 
<\lambda(3,1)<1$. 
\item{(2)} \ $0<\lambda(-3,-2)< \lambda(-1,-1) 
< \lambda(-2,-2)<1$. 
\item{(3)} \ $1<\lambda(0,-1)<\lambda(2,0) 
<\lambda(1,-1)<2$. 
\item{(4)} \ $1<\lambda(-3,-3)<\lambda(-1,-2)<2$. 
\item{(5)} \ $2<\lambda(-2,-3)<\lambda(0,-2)<3$. 
\end{description}
\end{Remark}

\begin{proof}
See Appendix. 
\end{proof}

Note that for nonzero vector $\wl=(l_{1},l_{2})$, $\tau^{*}(\wl)=(l_{2},-\lfloor\lambda(\wl)\rfloor-1)$. 
So by using Remark \ref{lambda}, all elements of $Q_{\beta}$ can be enumerated. Indeed, we find
\begin{equation}\label{state_set}
\begin{split}
Q_{\beta} = 
&\{ \pm(3,2), \pm(1,1), \pm(2,2), \pm(2,1), \pm(1,0), \pm(3,1), \pm(0,1), \\
&\ \pm(2,0), \pm(1,-1), \pm(3,3), \pm(1,2), \pm(2,3), \pm(0,2), (0,0)\} \\
\end{split}
\end{equation}
and, letting $(l_{1},l_{2}) \longrightarrow (l_{2},l_{3})$ denote  $\tau(l_{1},l_{2}) = (l_{2},l_{3})$, we have the following directed graph (Figure \ref{orbit}). 
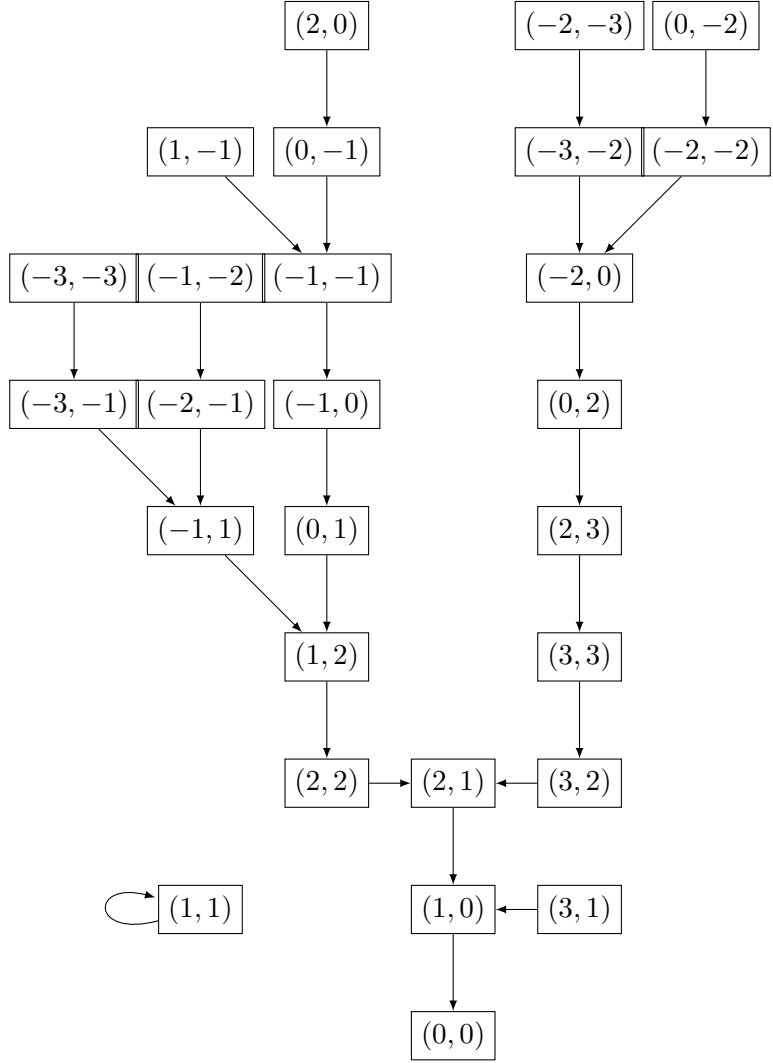
\begin{figure}[h]
\caption{all orbits by $\tau$ of each element of $Q_{\beta}$}\label{orbit}
\begin{center}
\scalebox{1.0}{
\begin{tikzpicture}[node distance = 1.68cm]

\node[rectangle,draw] (q0) {$(0,0)$} ; 
\node[rectangle,draw] (q1) [above of = q0] {$(1,0)$} ; 
\node[rectangle,draw] (q2) [above of = q1] {$(2,1)$} ; 
\node[rectangle,draw] (q3) [left of = q2] {$(2,2)$} ; 
\node[rectangle,draw] (q4) [above of = q3] {$(1,2)$} ; 
\node[rectangle,draw] (q5) [above of = q4] {$(0,1)$} ; 
\node[rectangle,draw] (q6) [above of = q5] {$(-1,0)$} ; 
\node[rectangle,draw] (q7) [above of = q6] {$(-1,-1)$} ; 
\node[rectangle,draw] (q8) [above of = q7] {$(0,-1)$} ; 
\node[rectangle,draw] (q9) [above of = q8] {$(2,0)$} ; 

\node[rectangle,draw] (p1) [right of = q1] {$(3,1)$} ; 

\node[rectangle,draw] (r1) [right of = q2] {$(3,2)$} ; 
\node[rectangle,draw] (r2) [above of = r1] {$(3,3)$} ; 
\node[rectangle,draw] (r3) [above of = r2] {$(2,3)$} ; 
\node[rectangle,draw] (r4) [above of = r3] {$(0,2)$} ; 
\node[rectangle,draw] (r5) [above of = r4] {$(-2,0)$} ; 
\node[rectangle,draw] (r6) [above of = r5] {$(-3,-2)$} ; 
\node[rectangle,draw] (r7) [above of = r6] {$(-2,-3)$} ; 

\node[rectangle,draw] (u1) [left of = q5] {$(-1,1)$} ; 
\node[rectangle,draw] (u2) [above of = u1] {$(-2,-1)$} ; 
\node[rectangle,draw] (u3) [above of = u2] {$(-1,-2)$} ; 

\node[rectangle,draw] (v1) [left of = u2] {$(-3,-1)$} ; 
\node[rectangle,draw] (v2) [left of = u3] {$(-3,-3)$} ; 

\node[rectangle,draw] (s1) [right of = r6] {$(-2,-2)$} ; 
\node[rectangle,draw] (s2) [above of = s1] {$(0,-2)$} ; 

\node[rectangle,draw] (t1) [left of = q8] {$(1,-1)$} ; 

\node[rectangle] (w1) [left of = q1] {} ; 
\node[rectangle,draw] (w2) [left of = w1] {$(1,1)$} ; 

\path[->,>=latex] 
(q1) edge [left] node [right] {} (q0)
(q2) edge [left] node [right] {} (q1)
(q3) edge [left] node [right] {} (q2)
(q4) edge [left] node [right] {} (q3)
(q5) edge [left] node [right] {} (q4)
(q6) edge [left] node [right] {} (q5)
(q7) edge [left] node [right] {} (q6)
(q8) edge [left] node [right] {} (q7)
(q9) edge [left] node [right] {} (q8)

(p1) edge [left] node [above] {} (q1)

(r1) edge [left] node [right] {} (q2)
(r2) edge [left] node [right] {} (r1)
(r3) edge [left] node [right] {} (r2)
(r4) edge [left] node [right] {} (r3)
(r5) edge [left] node [right] {} (r4)
(r6) edge [left] node [right] {} (r5)
(r7) edge [left] node [right] {} (r6)

(u1) edge [right] node [left] {} (q4)
(u2) edge [left] node [right] {} (u1)
(u3) edge [left] node [right] {} (u2)

(v1) edge [left] node [left] {} (u1)
(v2) edge [left] node [left] {} (v1)

(s1) edge [left] node [right] {} (r5)
(s2) edge [left] node [right] {} (s1)

(t1) edge [left] node [left] {} (q7)

(w2) edge [loop left] node [left] {} (w2)

 ;

\end{tikzpicture}
}
\end{center}
\end{figure}

\begin{Remark}\label{11}
$Q_{\beta} \setminus F_{\beta} = P_{\beta}=\{(1,1)\}$. 
\end{Remark}

\begin{proof}
By Figure \ref{orbit}. 
\end{proof}

\begin{proof}[Proof of Theorem \ref{main1}]
Recall 
$$V_{\beta} = \left\{ - \sum_{n=1}^{r} \omega_{n}\tau^{n-1}(0,1) \ \Biggl| \ r \in \mathbb{N}, \omega_{n} \in \mathbb{N}_{0} \ (n \in \lbc1,r\rbc) \right\}.$$
By Theorem \ref{main2} and Remark \ref{11}, 
it suffices to show that 
\begin{multicols}{2}
\begin{description}
\item{(i)} \ $\tau^{-1}(1,1) \subset \{(1,1)\}$. 
\item{(ii)} \ $R_{0}= \lbc-1,1\rbc^{2} \cap V_{\beta} \subset F_{\beta}$. 
\end{description}
\end{multicols}

(i) Let $\zeta \in \mathbb{Z}$ satisfy $\tau(\zeta,1)=(1,1)$ and let $g:=\lambda(\zeta,1)= t(\zeta-2)\beta^{-1}+t\beta^{-2}$. Then $-1 \leq g <0$. By Observation \ref{c_ineq} and Observation \ref{obs:4(-4)2}-(1), 
$$-1 \leq g < (t \zeta-2t+1)\beta^{-1} < t\zeta \beta^{-1}-1.$$
Thus we get $\zeta >0$. On the other hand, since 
$$0> g >t(\zeta-2)\beta^{-1},$$
we have $\zeta<2$. So $\zeta=1$ and hence the proof of (i) is completed. 

(ii) Since 
\begin{equation}\label{main1-1}
(0,1) \longrightarrow (1,2) \longrightarrow (2,2) \longrightarrow (2,1) 
\longrightarrow (1,0) \longrightarrow (0,0) 
\end{equation}
by Figure \ref{orbit}, we have $V_{\beta} \subset ((-\infty,0] \cap \mathbb{Z})^{2}$ and so 
$$R_{0} \subset \lbc-1,0\rbc^{2}.$$
Moreover by Figure \ref{orbit}, 
\begin{equation*}\label{main1-2}
(0,-1) \longrightarrow (-1,-1) 
\longrightarrow (-1,0) 
\longrightarrow (0,1).  
\end{equation*}
So by \eqref{main1-1}, $\lbc-1,0\rbc^{2} \subset F_{\beta}$. Hence $R_{0} \subset F_{\beta}$. 
\end{proof}

Since $\beta$ does not have property (PF), 
there is $x \in \mathbb{N}_{0}[1/\beta]$ such that 
$x$ does not have the finite beta-expansion. 
The following example is such $x$. 

\begin{Example}
Let 
$$x := (2t-2)+(2t-2)\beta^{-1}+(t-1)\beta^{-2}+(t-1)\beta^{-4}
+(t-1)\beta^{-5}+(t-1)\beta^{-6} \in \mathbb{N}_{0}[1/\beta].$$
Then $x$ does not have the finite beta-expansion. Indeed, 
by using $1000 \overset{\nu}{=}0(2t)(-2t)t$ 
and $2000\overset{\nu}{=}0(4t)(-4t)(2t)$, we have
\begin{equation*}
\scalebox{0.85}{$\displaystyle 
\begin{array}{cccccccccccccc}
d_{\beta}(\beta^{-2}x) 
&\overset{\nu}{=} 0 & (2t-2) & (2t-2) & (t-1) & 0 & (t-1) & (t-1) & (t-1) & 0 & 0 & 0 & 0 & 0^{\infty} \\
&\overset{\nu}{=} 1 & (-2) & (4t-2) & (-1) & 0 & (t-1) & (t-1) & (t-1) & 0 & 0 & 0 & 0 & 0^{\infty} \\
&\overset{\nu}{=} 1 & 0 & (-2) & (4t-1) & (-2t) & (t-1) & (t-1) & (t-1) & 0 & 0 & 0 & 0 & 0^{\infty} \\
&\overset{\nu}{=} 1 & 0 & 0 & (-1) & (2t) & (-t-1) & (t-1) & (t-1) & 0 & 0 & 0 & 0 & 0^{\infty} \\
&\overset{\nu}{=} 1 & 0 & 0 & 0 & 0 & (t-1) & (-1) & (t-1) & 0 & 0 & 0 & 0 & 0^{\infty} \\
&\overset{\nu}{=} 1 & 0 & 0 & 0 & 0 & (t-2) & (2t-1) & (-t-1) & t & 0 & 0 & 0 & 0^{\infty} \\
&\overset{\nu}{=} 1 & 0 & 0 & 0 & 0 & (t-2) & (2t-2) & (t-1) & (-t) & t & 0 & 0 & 0^{\infty} \\
&\overset{\nu}{=} 1 & 0 & 0 & 0 & 0 & (t-2) & (2t-2) & (t-2) & t & (-t) & t & 0 & 0^{\infty} \\
&\overset{\nu}{=} 1 & 0 & 0 & 0 & 0 & (t-2) & (2t-2) & (t-2) & (t-1) & t & (-t) & t & 0^{\infty} \\
\end{array}
$} 
\end{equation*}
Thus 
\begin{equation}\label{eq:ex1}
d_{\beta}(\beta^{-2}x) \overset{\nu}{=} 10 000 (t-2)(2t-2)(t-2)(t-1)^{\infty}. 
\end{equation}
Now, we can verify $d_{\beta}(1) = (2t-2)(2t-2)(t-1)00t0^{\infty}$. So 
$$d_{\beta}^{*}(1) = ((2t-2)(2t-2)(t-1)00(t-1))^{\infty}.$$
Therefore by Theorem \ref{parry}, 
the right hand side of \eqref{eq:ex1} 
belongs to $D_{\beta}$. Hence 
$$\mbox{the beta-expansion of $x$} 
= 10.000(t-2)(2t-2)(t-2)(t-1)^{\infty}.$$
\end{Example}

Theorem \ref{main2} also provides some other examples of $\beta$ with property (F$_{1}$) without (PF). 

\begin{Example}
Let $\beta>1$ be an algebraic 
integer with minimal polynomial $x^{3}-ax^{2}-bx-c$. 
Suppose that $(a,b,c) \in \{(5,-5,3),(6,-6,4),(7,-8,5)\}$. 
Then by computer calculations, we can check 
\begin{equation*}
\# Q_{\beta} = \left\{ 
\begin{array}{ll}
43 & \mbox{if $(a,b,c)=(5,-5,3)$} \\
67 & \mbox{if $(a,b,c)=(6,-6,4)$} \\
117 & \mbox{if $(a,b,c)=(7,-8,5)$} \\
\end{array}
\right. 
\end{equation*}
and 
$$P_{\beta}=\{(1,1)\}.$$
Moreover in these cases, we can verify that all orbits by $\tau$ of $(0,-1)$ are the same as the $\beta$ of Theorem \ref{main1}. That is, 
\begin{equation*}
(0,-1) \rightarrow (-1,-1) \rightarrow (-1,0) \rightarrow 
(0,1) \rightarrow (1,2) \rightarrow (2,2) \rightarrow (2,1) 
\rightarrow (1,0) \rightarrow (0,0). 
\end{equation*}
Therefore $R_{0}=\lbc-1,1\rbc \cap V_{\beta} \subset \lbc-1,0\rbc^{2} \subset F_{\beta}$ and so $\beta$ satisfies the conditions of Theorem \ref{main2}. 
Hence in each case, $\beta$ has property (F$_{1}$). 
\end{Example}

\subsection{Proof of Proposition \ref{CPcase}}\label{subsec:CPcase}

In this subsection, let $\beta>1$ be a cubic Pisot number with minimal polynomial 
$$x^{3}-ax^{2}-bx-c.$$
Frougny and Solomyak proved the following. 

\begin{Thm}[(\cite{frougny1})]\label{charaPF}
Let $\beta$ be a Pisot number with property (PF). 
Then $\beta$ has property (F) if and only if 
$d_{\beta}(1)$ is finite. 
\end{Thm}

By Theorem \ref{charaPF}, if $\beta$ has property (PF) without (F), 
then $d_{\beta}(1)$ is not finite. 
So we want to show that if $d_{\beta}(1)$ is not finite, 
then $\beta$ has property (PF) without (F) or does not have property (F$_{1}$). 

Bassino determined $d_{\beta}(1)$ for each cubic Pisot number $\beta$. 
In \cite{bassino}, it is shown that $d_{\beta}(1)$ is not finite 
if and only if $(a,b,c)$ satisfies one of the followings: 
\begin{description}
\item{(I)} \ $0 < b \leq a$ and $c<0$ 
\item{(II)} \ $-a < b \leq 0$ and $b+c <0$ 
\item{(III)} \ $b \leq -a$ and 
$b(k-1)+c(k-2) > (k-2)-(k-1)a$ where $k \in \lbc2,a-2\rbc$ such that, 
denoting $e_{k}=1-a+(a-2)/k$, $e_{k} \leq b+c < e_{k-1}$. 
\end{description}
Moreover in \cite{bassino}, we see that 
\begin{equation}\label{floor_beta_val}
\lfloor \beta \rfloor = \left\{ 
\begin{array}{ll}
a & \mbox{if $(a,b,c)$ satisfies (I)} \\
a-1 & \mbox{if $(a,b,c)$ satisfies (II)} \\
a-2 & \mbox{if $(a,b,c)$ satisfies (III).} \\
\end{array}
\right. 
\end{equation}
So, since $\lfloor \beta \rfloor = a+ \lfloor b\beta^{-1}+c\beta^{-2} \rfloor$, 
we have 
\begin{equation}\label{eq:init_val}
\lambda(\wl_{I}) = b\beta^{-1}+c\beta^{-2} \in \left\{ 
\begin{array}{ll}
(0,1) & \mbox{if $(a,b,c)$ satisfies (I)} \\
(-1,0) & \mbox{if $(a,b,c)$ satisfies (II)} \\
(-2,-1) & \mbox{if $(a,b,c)$ satisfies (III).} 
\end{array}
\right. 
\end{equation}
Here, recall 
$$\wl_{I}=(0,1).$$

\begin{Lem}\label{CPcase_lem}
 \ 
\begin{description}
\item{(1)} \ If $(a,b,c)$ satisfies (I), (II) or (III), then $\beta \geq 2$. 
\item{(2)} \ Let $(a,b,c)$ satisfy (I), (II) with $c>0$, or (III). 
Then $\lfloor \beta \rfloor +1$ does not have the finite beta-expansion. 
\end{description}
\end{Lem}

In order to prove Lemma \ref{CPcase_lem}, 
we need the following remark. 

\begin{Remark}\label{CPcase_lambda}
 \ 
\begin{description}
\item{(1)} \ Suppose that $(a,b,c)$ satisfies (I). Then 
$0<\lambda(\wl_{I})< \lambda(-1,1)<1$. 
\item{(2)} \ Suppose that $(a,b,c)$ satisfies (II) with $c>0$. Then 
$-1<\lambda(\wl_{I}) < 0 < \lambda(1,0)<1$. 
\item{(3)} \ Suppose that $(a,b,c)$ satisfies (III). Then 
$-2< \lambda(\wl_{I})<-1< \lambda(1,1) < 0$. 
\end{description}
\end{Remark}

\begin{proof}
See Appendix. 
\end{proof}

\begin{proof}[Proof of Lemma \ref{CPcase_lem}]
(1) It suffices to show $\lfloor \beta \rfloor \geq 2$. 
We first consider the case (I). 
By Lemma \ref{cp}, 
$$0 \leq b-1 \leq a+c-1 \leq a-2 =\lfloor \beta \rfloor-2.$$
Thus $2 \leq \lfloor \beta \rfloor$. 
Next, consider the case (II). 
Since $1-b < a+c$ by Lemma \ref{cp} and $b+c<0$, we have 
$$2 \leq a+b+c \leq a-1 = \lfloor \beta \rfloor.$$
Finally, consider the case (III). Note $c \leq \lfloor \beta \rfloor$ by Observation \ref{c_ineq}. So by \eqref{floor_beta_val}, 
$c \leq \lfloor \beta \rfloor = a-2 \leq -b-2$. 
Thus $b+c\leq -2$. 
Therefore we have 
$$2 \leq a+b+c \leq a-2 = \lfloor \beta \rfloor.$$

(2) By (1) and Remark \ref{floor_beta+1}, 
we want to show $-\wl_{I} \notin F_{\beta}$. 
Notice that $\wl_{I} \notin F_{\beta}$ by Remark \ref{T1_exp}. 
So it suffices to show 
$$\tau^{k}(-\wl_{I}) = \wl_{I} \ \mbox{for some $k$}.$$

$Case$ $(I)$. By Remark \ref{CPcase_lambda}-(1), 
$-\wl_{I} \longrightarrow (-1,1) 
\longrightarrow (1,0) 
\longrightarrow \wl_{I}$. 

$Case$ $(II)$ with $c>0$. By Remark \ref{CPcase_lambda}-(2), 
$-\wl_{I} \longrightarrow (-1,0) 
\longrightarrow \wl_{I}$. 

$Case$ $(III)$. By Remark \ref{CPcase_lambda}-(3), 
$-\wl_{I} \longrightarrow (-1,-1) 
\longrightarrow (-1,0) 
\longrightarrow \wl_{I}$. 
\end{proof}

\begin{proof}[Proof of Proposition \ref{CPcase}]
It suffices to show that $\beta$ has property (PF) without (F) or does not have property (F$_{1}$) in case (I), (II) or (III). Now in case (I), (III) or (II) with $c>0$, $\lfloor \beta \rfloor+1$ does not have the finite beta-expansion by Lemma \ref{CPcase_lem}-(2) and so $\beta$ does not have property (F$_{1}$). Consider the case (II) with $c<0$. Note that $1-b-c < a$ by Lemma \ref{cp}. So by Theorem \ref{pf}, $\beta$ has property (PF) without (F). 
\end{proof}

\subsection{No cubic Pisot unit which has property (F$_{1}$) without (PF)}

First, Akiyama characterized cubic Pisot units with property (F). 

\begin{Thm}[(\cite{akiyama1})]\label{cpuF}
Let $\beta$ be a cubic Pisot unit. 
Then $\beta$ has property (F) if and only if $\beta$ is a 
root of the following polynomial with integer coefficients: 
$$x^{3}-ax^{2}-bx-1, \ a \geq 0 \ \mbox{and} \ -1 \leq b \leq a+1.$$
\end{Thm}

\begin{Remark}\label{cor:cpuF}
Let $\beta$ be a cubic Pisot unit with minimal polynomial 
$$x^{3}-ax^{2}-bx-c.$$
Then $\beta$ has property (F) if and only if $c=1$ and $b+c \geq 0$. 
\end{Remark}

\begin{proof}
If $\beta$ has property (F), then $c=1$ and $b+c=b+1 \geq 0$ by Theorem \ref{cpuF}. 
Now we want to show the opposite. Suppose that $c=1$ and $b+c \geq 0$. 
Then $b \geq -c=-1$. Note that $-a+1 \leq b \leq a+1$ by Lemma \ref{cp}. 
So we have $-1 \leq b \leq a+1$ and $a \geq 0$. 
Hence by Theorem \ref{cpuF}, $\beta$ has property (F). 
\end{proof}

In summary, we get the following corollary. 

\begin{Cor}
Let $\beta$ be a cubic Pisot unit with minimal polynomial 
$$x^{3}-ax^{2}-bx-c.$$
Then 
\begin{description}
\item{(1)} \ $\beta$ has property (F) if and only if $b+c \geq 0$ and $c=1$. 
\item{(2)} \ The following conditions are equivalent. 
\begin{description}
\item{(i)} \ $\beta$ has property (PF). 
\item{(ii)} \ $\beta$ has property (F$_{1}$). 
\item{(iii)} \ $(b+c)c \geq 0$ and $(b,c) \neq (1,-1)$. 
\end{description}
\end{description}
\end{Cor}

\begin{proof}
By Remark \ref{cor:cpuF}, (1) holds. We want to show (2). Now it suffices to show that 
\begin{description}
\item{(I)} \ If $(b+c)c \geq 0$ and $(b,c) \neq (1,-1)$, then $\beta$ has property (PF). 
\item{(II)} \ If $(b+c)c < 0$ or $(b,c)=(1,-1)$, then $\beta$ does not have property (F$_{1}$). 
\end{description}

(I) Suppose that $(b+c)c \geq 0$ and $(b,c) \neq (1,-1)$. Since $c=1$ implies $b+c \geq 0$, it suffices to consider the case $c=-1$. Note that $b \leq 0$ because $b+c \leq 0$ and $b\neq 1$. So, since $1-b-c<a$ by Lemma \ref{cp}, $\beta$ has property (PF) without (F) by Theorem \ref{pf}. 

(II) Suppose that $(b+c)c<0$ or $(b,c)=(1,-1)$. 
First, consider the case $c=1$. In this case, $\beta$ does not have property (PF) by Theorem \ref{pf}. Moreover we see that $b<0$ because $b+c<0$. Therefore, since $-a<b$ by Lemma \ref{cp}, $(a,b,c)$ satisfies (II) and so $d_{\beta}(1)$ is not finite. 
Next, consider the case $c=-1$. Then $b > 0$ because $b+c \geq 0$. 
Therefore, since $b<a$ by Lemma \ref{cp}, $(a,b,c)$ satisfies (I) and so $d_{\beta}(1)$ is not finite. Hence by Proposition \ref{CPcase}, $\beta$ does not have property (F$_{1}$). 
\end{proof}

\section*{Appendix: Proof of Remarks}

\begin{proof}[Proof of Remark \ref{lambda}]
Notice that 
\begin{equation}\label{eq:lambda1}
\lambda(-1,-1)>0 
\end{equation}
because $\lambda(-1,-1) 
= t\beta^{-1}-t\beta^{-2}=t\beta^{-1}(1-\beta^{-1})$. 

(1) First, we show $\lambda(3,1)<1$. 
Let $x:=1-\lambda(3,1)=1-t\beta^{-1}-t\beta^{-2}$. Then by \eqref{eq:obs-2}, 
\begin{equation*}
x = t\beta^{-1}-3t\beta^{-2}+t\beta^{-3} 
= (t-1)\beta^{-1}-t\beta^{-2}-t\beta^{-3}+t\beta^{-4} 
= (t-2)\beta^{-1}+x\beta^{-1}+t\beta^{-4}. 
\end{equation*}
So, since $x(1-\beta^{-1})=(t-2)\beta^{-1}+t\beta^{-4}>0$, 
we have $x > 0$. Thus 
$$\lambda(3,1) =1-x < 1.$$
Note that $\lambda(2,1) = t\beta^{-2}>0$. So 
$$\lambda(3,1) = \lambda(1,0) + \lambda(2,1) 
> \lambda(1,0).$$
By \eqref{eq:lambda1}, 
$$\lambda(2,1) = \lambda(1,0) + \lambda(-1,-1) 
< \lambda(1,0).$$

(2) Since $\beta>t\geq 2$, we have 
$$\lambda(-3,-2) 
= t\beta^{-1}-2t\beta^{-2} 
= t\beta^{-1}(1-2\beta^{-1}) > 0.$$
By (1), 
$$\lambda(-3,-2) 
= \lambda(-1,-1) - \lambda(2,1) 
< \lambda(-1,-1).$$
By \eqref{eq:lambda1}, 
$$\lambda(-2,-2) 
= \lambda(-1,-1) + \lambda(-1,-1) 
> \lambda(-1,-1).$$
So by \eqref{eq:obs-2}, 
$$\lambda(-2,-2) 
= 2t\beta^{-1}-2t\beta^{-2} 
= 1-t\beta^{-3} <1.$$

(3) Since $\beta<2t-1$ by Observation \ref{obs:4(-4)2}-(1), we have 
$$\lambda(0,-1) = 2t\beta^{-1}-t\beta^{-2} 
> (2t-1)\beta^{-1} > 1.$$
By (1), 
$$\lambda(2,0)= \lambda(0,-1)+\lambda(2,1) 
>\lambda(0,-1).$$
By \eqref{eq:lambda1}, 
$$\lambda(2,0) = \lambda(1,-1) - \lambda(-1,-1) 
< \lambda(1,-1).$$
By \eqref{eq:obs-2}, 
\begin{equation*}
\begin{split}
\lambda(1,-1) &= 3t\beta^{-1}-t\beta^{-2} \\
&= 2 - t\beta^{-1}+3t\beta^{-2}-2t\beta^{-3} \\
&= 2 - (t-1) \beta^{-1} + t\beta^{-2}-t\beta^{-4} \\
&< 2-(t-2) \beta^{-1} \\
&\leq 2. \\
\end{split}
\end{equation*}

(4) By (2) and (3), 
$$\lambda(-3,-3) = \lambda(-3,-2) + \lambda(0,-1)>1.$$
By (1), 
$$\lambda(-3,-3) 
= \lambda(-1,-2) - \lambda(2,1) 
< \lambda(-1,-2).$$
By \eqref{eq:obs-2}, 
$$\lambda(-1,-2) 
= 3t\beta^{-1}-2t\beta^{-2} 
= 1+t\beta^{-1}-t\beta^{-3} < 2.$$

(5) Since $\lambda(0,-2)<1+\lambda(2,0)$ by (2), we have by (3), 
$$\lambda(0,-2) < 1+\lambda(2,0) <3.$$
By (1), 
$$\lambda(-2,-3) 
= \lambda(0,-2) - \lambda(2,1) 
< \lambda(0,-2).$$
Since $2\beta^{-1} \leq t\beta^{-1}<1$, we have by \eqref{eq:obs-2}, 
\begin{equation*}
\lambda(-2,-3) -2 
= 4t\beta^{-1}-3t\beta^{-2}-2 
=t\beta^{-2}-2t\beta^{-3} = t\beta^{-2}(1-2\beta^{-1})>0. 
\end{equation*}
\end{proof}

\begin{proof}[Proof of Remark \ref{CPcase_lambda}]
(1) By \eqref{eq:init_val}, $\lambda(\wl_{I})>0$. 
Moreover, since $c<0$, we get 
$$\lambda(\wl_{I}) < -c\beta^{-1}+\lambda(\wl_{I}) 
= \lambda(-1,1) .$$
Now, we show $\lambda(-1,1)<1$. 
Since $b-1 < a+c$ by Lemma \ref{cp}, we have 
$b-c\leq a= \lfloor \beta \rfloor$ by \eqref{floor_beta_val}. So 
$$\lambda(-1,1) = (b-c)\beta^{-1}+c\beta^{-2} 
< \lfloor \beta \rfloor \beta^{-1} <1.$$

(2) By \eqref{eq:init_val}, $-1<\lambda(\wl_{I})<0$. 
Since $c>0$ and $c < \beta$ by Observation \ref{c_ineq}, 
$0<\lambda(1,0)=c\beta<1$. 

(3) By \eqref{eq:init_val}, $-2<\lambda(\wl_{I})<-1$. 
Since $c<\beta$ by Observation \ref{c_ineq}, 
we have $\lambda(1,0)<1$ and so 
$$\lambda(1,1) < 1+\lambda(\wl_{I}) <0.$$
Note that $b+c \geq 2-a$ by Lemma \ref{cp}. So $c >0$ beause $b \leq -a$. 
Moreover by \eqref{floor_beta_val}, 
$$\lambda(1,1) = (b+c)\beta^{-1}+c\beta^{-2} > -(a-2)\beta^{-1} = -\lfloor \beta \rfloor \beta^{-1} >-1.$$
\end{proof}

Fumichika Takamizo

OCAMI, Osaka Metropolitan University, 

3-3-138 Sugimoto, Sumiyoshi-ku Osaka, 558-8585, Japan

\end{document}